\documentclass{article}       %
\RequirePackage{fix-cm}
\usepackage[normalem]{ulem}

\usepackage{graphicx}

\usepackage{color}
\usepackage{amssymb}
\usepackage{amsmath}
\usepackage{amsthm}
\usepackage{mathrsfs}
\usepackage[authoryear]{natbib}

\usepackage{tikz}
\usepackage{hyperref}
\usepackage{xcolor}\hypersetup{allcolors=blue,colorlinks=true}

\newcommand{\R}{\mathbb{R}}

\newcommand{\E}{\mathbf{E}}
\renewcommand{\P}{\mathbf{P}}
\newcommand{\I}{\mathbb{I}}

\newcommand{\rvin}[1]{{\color{red}{#1}}}
\renewcommand{\rvin}[1]{#1}
\newcommand{\rvout}[1]{{\color{red}{\sout{#1}}}}
\renewcommand{\rvout}[1]{}

\newcommand{\cf}[1]{\\\fbox{\begin{minipage}{0.95\linewidth}{\large\textup{\texttt{[SUPPLEMENT]}}} #1\end{minipage}}\\} 
\renewcommand{\cf}[1]{} 

\newcommand{\what}[1]{{{#1}}}

\newcommand\hcancel[2][red]{\setbox0=\hbox{$#2$}%
\rlap{\raisebox{.25\ht0}{\textcolor{#1}{\rule{\wd0}{1pt}}}}#2} 
\renewcommand\hcancel[2][red]{}

\newcommand{\vertiii}[1]{{\left\vert\kern-0.25ex\left\vert\kern-0.25ex\left\vert #1 
    \right\vert\kern-0.25ex\right\vert\kern-0.25ex\right\vert}}

\numberwithin{equation}{section}
\theoremstyle{plain}

\newtheorem{theorem}{Theorem}[section]

\newtheorem{proposition}{Proposition}[section]

\newtheorem{remark}{Remark}

\newtheorem{claim}{Claim}

\begin{document}

\title{Lyapunov Conditions for Differentiability of Markov Chain Expectations: the Absolutely Continuous Case}


\author{Chang-Han Rhee$^{1}$
        \and
        Peter Glynn$^{2}$
}
\footnotetext[1]{Stochastics Group, Centrum Wiskunde \& Informatica, Science Park 123, Amsterdam, North Holland, 1098 XG, The Netherlands}
\footnotetext[2]{Management Science \& Engineering, Stanford University, 500 W. 120 Street, New York, New York 10027, USA
}



\date{\today}

\maketitle

\begin{abstract}
We consider a family of Markov chains whose transition dynamics are affected by model parameters.
Understanding the parametric dependence of (complex) performance measures of such Markov chains is often of significant interest. 
The derivatives of the performance measures w.r.t.\ the parameters play important roles, for example, in numerical optimization of the performance measures, and quantification of the uncertainties in the performance measures when there are uncertainties in the parameters from the statistical estimation procedures.
In this paper, we establish conditions that guarantee the differentiability of various types of intractable performance measures---such as the stationary and random horizon discounted performance measures---of general state space Markov chains and provide probabilistic representations for the derivatives. 


\end{abstract}

\section{Introduction}\label{sec:intro}
Let $X=(X_n: n\geq 0)$ be a Markov chain taking values in a state space $S$. For the purpose of this paper, the state space $S$ may be discrete or continuous. In many applications settings, it is natural to consider the behavior of $X$ as a function of a parameter $\theta$ that affects the transition dynamics of the process. In particular, suppose that for each $\theta$ in some open neighborhood of $\theta_0\in \R^d$, $P(\theta) = (P(\theta, x, dy): x,y\in S)$ defines the one-step transition kernel of $X$ associated with parameter choice $\theta$. In such a setting, computing the derivative of some application-specific expectation is often of interest. 

Such derivatives play a key role when one is numerically optimizing an objective function, defined as a Markov chain's expected value, over the decision parameter $\theta$. In addition, such derivatives describe the sensitivity of the expected value under consideration to perturbations in $\theta$. Such sensitivities are valuable in statistical applications, and arise when one applies (for example) the ``delta method'' in conjunction with estimating equations involving some expectation of the observed Markov chain; see, for example, \cite{lehmann06}. More generally, sensitivity analysis is important when one is interested in understanding how robust the model is to uncertainties in the input parameters.

In particular, suppose that $\theta$ is a vector of statistical parameters, and that a data set of size $n$ has been collected to estimate the underlying true parameter $\theta^*$. In significant generality, the associated estiamtor $\hat \theta_n$ for $\theta^*$ will satisfy a central limit theorem (CLT) of the form 
$$
n^{1/2} (\hat \theta_n - \theta^*) \Rightarrow N(0,C)
$$
as $n\to \infty$, where $\Rightarrow$ denotes weak convergence and $N(0,C)$ is a normally distributed random column vector with mean $0$ and covariance matrix $C$; see, for example, \cite{ibragimov81}. 
In many applications, one wishes to understand how the uncertainty in our estimator $\hat \theta_n$ of $\theta^*$ propagates through the model associated with $X$ to produce uncertainty in output measures of interest. 
Suppose, for example, that the decision-maker focuses her attention on a performance measure of the form $\alpha(\theta) = \E_\theta Z$, where $Z$ is some appropriately chosen random variable (rv) and $\E_\theta(\cdot)$ is the expectation operator under which $X$ evolvoes according to $\P(\theta)$. 
If $\alpha(\cdot)$ is differentiable at $\theta^*$, then
$$
n^{1/2}(\alpha(\hat \theta_n)-\alpha(\theta^*)) \Rightarrow \nabla \alpha(\theta^*) N(0,C)
$$
as $n\to \infty$, where $\nabla\alpha(\theta)$ is the (row) gradient vector evaluated at $\theta$; see p.122 of \cite{serfling80}. If, in addition, $\nabla \alpha(\cdot)$ is continuous at $\theta^*$ and $C$ can be consistently estimated from the observed data via an estimator $C_n$, the interval 
\begin{equation}\label{eqn_1_1}
\left[\alpha(\hat \theta_n) - z\frac{\sigma_n}{\sqrt{n}},\ \alpha(\hat \theta_n) + z\frac{\sigma_n}{\sqrt{n}}\right]
\end{equation}
is an asymptotic $100(1-\delta)$\% confidence interval for $\alpha(\theta^*)$ (provided $\nabla \alpha(\theta^*)C\nabla \alpha(\theta^*)^T > 0$), where $z$ is chosen so that $P(-z\leq N(0,1)\leq z) = 1-\delta$ and $\sigma_n = \sqrt{\nabla \alpha(\hat \theta_n)C_n \nabla \alpha(\hat \theta_n)^T}$. The confidence interval (\ref{eqn_1_1}) provides the modeler with the desired sensitivity and robustness of the model described by $X$ to the statistical uncertainties present in the estiamtion of $\theta^*$. Of course, this approach rests on the differentiability of $\alpha(\cdot)$ and on one's ability to compute the gradient. This paper provides conditions guaranteeing differentiability in the general state space Markov chian settings and provides representations for those derivatives suitable for computation. 

The problem of determining such differentiability has a long history and has been addressed through various approaches including weak differentiation (\citealp{vazquez1992, pflug1992}), likelihood ratio \citep{glynn95}, measure-valued differentiation (\citealp{heidergott06}, \linebreak\citealp{heidergott2006B}), and derivative regenration (\citealp{glasserman1993}). 
However, most of the previous approaches are limited to special classes of problems. For example, the results in \linebreak\cite{vazquez1992} and \cite{pflug1992} are limited to bounded performance functionals; 
\cite{glasserman1993} impose special structures in the the transition dynamics of the Markov chains and their parametrization;
\cite{glynn95} assumes for random horizon expectations that the associated stopping times have finite exponential moments, and for stationary expectations that the Markov chain is geometrically ergodic.
\cite{heidergott2006B} provide weaker conditions for random horizon performance measures based on measure-valued differentiation approach, but their sufficient conditions are difficult to verify in general and still require that the associated stopping times possess finite (at least) second moment.
Also based on measure-valued differentiation, \cite{heidergott06} study stationary expectations. However, the sufficient conditions verifiable based on the model building blocks in the paper require geometric ergodicity of the Markov chain.
In this paper, on the other hand, we provide easily verifiable sufficient conditions for random horizon expectations that do not require any moment conditions for the associated stopping times---hence, allowing even infinite horizon expectations. 
For stationary expectations, we provide (again, easily verifiable) sufficient conditions that does not require geometric ergodicity. 
We illustrate the sharpness of our differentiability criteria with the example of waiting times of G/G/1 queues with heavy tailed service times.

The rest of the paper is organized as follows. 
Section 2 develops a preliminary theory for both random-horizon expectations and stationary expectations based on simple and clean operator theoretic arguments. 
Section 3 provides more general criteria for differentiability of random horizon expectations based on the stochastic Lyapunov type inequalities arguments.
In Section 4, also taking Lyapunov inequalities approach, we establish the differentiability criteria for stationary expectations. 

%
%
\section{Operator-theoretic Criteria for Differentiability}\label{sec:operator}

We start by studying differentiability in a setting in which one can use operator arguments to establish existance of derivatives. In this operator setting, the proofs and theorem statements are especially straightforward.

Consider a Markov chain $X=(X_n:n\geq 0)$ living on state space $S$, with one step transition kernel $P=(P(x,dy): x,y\in S)$, where
$$
	P(x,dy) = P(X_{n+1}\in dy|X_n = x)
$$
for $x,y\in S$. We focus first on expectations of the form 
\begin{equation}\label{def2:u_star}
	u^*(x) = \E_x \sum_{j=0}^{T-1} \exp(\sum_{k=0}^{j-1} g(X_k)) f(X_j) + \exp(\sum_{k=0}^{T-1}g(X_k)) f(X_T),
\end{equation}
where $T = \inf\{n\geq 0: X_n \in C^c\}$ is the first hitting time of the ``target set'' $C^c\subseteq U$, $f:S\to\R_+$, $g:S\to \R$, and $\E_x(\cdot)\triangleq \E(\cdot|X_0=x)$. 

In (\ref{def2:u_star}), we permit the possibility that $C^c = \phi$, in which case $T=\infty$ a.s., and $u^*$ is then to be interpreted as the ``infinite horizon discounted reward''
\begin{align*}
u^*(x) = \E _x \sum_{j=0}^\infty \exp\left(\sum_{k=0}^{j-1} g(X_k)\right) f(X_j).
\end{align*}
In addition to subsuming infinite horizon discounted rewards, (\ref{def2:u_star}) also includes expected hitting times ($g\equiv 0, f=1$ on $C$ and $f=0$ on $C^c$), exit probabilities ($g\equiv 0$, $f=0$ on $C$, and $f(x) = I(x\in B)$ for $x\in C^c$, when one is considering $P(X_T\in B | X_0 = x)$), and many other natural Markov chain expectations. 

It is easy to verify that 
\begin{equation}\label{eq:u_star}
u^* = \sum_{n=0}^\infty K^n \tilde f
\end{equation}
where $K = (K(x,dy): x, y\in C)$ is the non-negative kernel for which 
\begin{equation}\label{eq2:2.3}
K(x,dy) = \exp(g(x)) P(x,dy)
\end{equation}
for $x,y\in S$, and 
\begin{equation*}
\tilde f(x) = f(x) + \int_{C^c} \exp(g(x)) P(x,dy)f(y)
\end{equation*}
for $x \in C$. Here, we are taking advantage in (\ref{def2:u_star}) of the (common) notational convention that for a function $h:B \to \R$, a measure $\eta$ on $B$, and kernels $Q_1$ and $Q_2$ on $B$, the scalar $\eta h$, the function $Q_1 h$, the measure $\eta Q_1$, and the kernel $Q_1 Q_2$ are respectively defined via

	$$\eta h = \int_B h(y) \eta(dy),$$

	$$(Q_1 h)(x) = \int_B h(y) Q_1(x,dy),$$

	$$(\eta Q_1)(A) = \int_B \eta(dx) Q_1 (x,A),$$

	$$(Q_1 Q_2)(x,A) = \int_B Q_1(x,dy) Q_2(y,A),$$
whenever the right-hand sides are well-defined. Furthermore, we define the kernels $Q^n$ via  $Q^0 (x,dy) = \delta_x(dy)$ (where $\delta_x( . )$ is a unit point mass at $x$), and $Q^n = Q Q^{(n-1)}$ for $n\geq1$.

Our goal is to use operator-theoretic tools to study the differentiability of (\ref{eq:u_star}). To this end, we start by defining the appropriate linear spaces that underlie this approach. Given a measurable space $(B, \mathcal B)$, measurable $w:B\to[1,\infty)$, and $h:B\to\R$, let $\|h\|_w = \sup\{ |h(x)|/w(x): x\in B\}$ and $L_w=\{h\in L:\|h\|_w < \infty\}$ where $L$ is the set of measurable functions. For a linear operator $Q:L_w\to L_w$ and a functional $\eta:L_w\to \R$, set
$$\vertiii{Q}_w = \sup_{h\in L_w: \|h\|_w \neq 0} \frac{\|Qh\|_w}{\|h\|_w}$$ 
and
$$\|\eta\|_{w} = \sup\{|\eta h|: h\in L_w, \|h\|_w \leq 1\}.$$
Then, let , $\mathcal L_w = \{Q\in \mathcal L: \vertiii{Q}_w < \infty\}$, and $\mathcal M_w = \{\eta\in \mathcal M: \|\eta\|_w < \infty\}$ where $\mathcal L$ and $\mathcal M$ are the sets of kernels, and measures, respectively. 
Each of the spaces $L_w$, $\mathcal L_w$, and $\mathcal M_w$ are Banach spaces under their respective norms and addition / scalar multiplication operations; see Appendix~\ref{appendix:B}.
Furthermore, for $Q_1, Q_2 \in \mathcal L_w$, $h\in L_w$, and $\eta \in \mathcal M_w$, it is easy to show that
\begin{equation}\label{eq:2.4}
\vertiii{Q_1Q_2}_w \leq \vertiii{Q_1}_w\cdot\vertiii{Q_2}_w
\end{equation}
and
\begin{equation}\label{eq:2.5}
\begin{aligned}
\|Qh\|_w &\leq \vertiii{Q}_w \cdot\|h\|_w,\\
\|\eta Q\|_w &\leq \|\eta\|_w \cdot \vertiii{Q}_w,\\
|\eta h| &\leq \|\eta\|_w \cdot \|h \|_w;
\end{aligned}
\end{equation}
see, for example, \cite{Dunford71} for the special case $w\equiv 1$. 
In view of (\ref{eq:2.4}), if $\vertiii{Q^m}_w < 1$ for some $m\geq 1$, then $(I-Q)$ is invertible on $\mathcal L_w$ and 
$$ (I-Q)^{-1} = \sum_{n=0}^\infty Q^n.$$

Given a parametrized family of kernels $(Q(\theta)\in \mathcal L_w: \theta \in (a,b))$, we say that $Q(\cdot)$ is \emph{continuous} in $\mathcal L_w$ at $\theta_0 \in (a,b)$  if 
$\vertiii{Q(\theta_0 + h)-Q(\theta_0) }_w \to 0$
as $h\to 0$,  
and  \emph{differentiable} in $\mathcal L_w$ at $\theta_0 \in (a,b)$ with derivative $Q'(\theta_0)$ if there exists a kernel $Q'(\theta_0) \in \mathcal L_w$ for which 
$$\vertiii{\frac{Q(\theta_0 + h)-Q(\theta_0)}{h} - Q'(\theta_0)}_w \to 0$$
as $h\to 0$. 
If $Q(\cdot)$ is differentiable in a neighborhood of $\theta_0$ with derivative $Q'(\cdot)$, and $Q'(\cdot)$ is continuous in $\mathcal L_w$, then we say that $Q(\cdot)$ is continuously differentiable at $\theta_0$.
Similarly, given families $(f(\theta)\in L_w: \theta\in (a,b))$ and $(\eta(\theta) \in \mathcal M_w: \theta\in (a,b))$, we say that $f(\cdot)$ is 
\emph{continuous} in $ L_w$ at $\theta_0$ if 
$\left\|f(\theta_0+h) - f(\theta_0)\right\|_w \to 0$
as $h\to 0$, 
and \emph{differentiable} in $ L_w$ at $\theta_0$ if there exists $f'(\theta_0)\in L_w$ such that 
$$\left\|\frac{f(\theta_0+h) - f(\theta_0)}{h} - f'(\theta_0)\right\|_w \to 0$$
as $h\to 0$; and $\eta(\cdot)$ is 
\emph{continuous} in $\mathcal M_w$ at $\theta_0$ if 
$ \left\|\eta(\theta_0 + h) - \eta(\theta_0)\right\|_w \to 0$
as $h\to0$,
and differentiable in $\mathcal M_w$ at $\theta_0$ if there exists $\eta'(\theta_0) \in \mathcal M_w$ such that 
$$ \left\|\frac{\eta(\theta_0 + h) - \eta(\theta_0)}{h} - \eta'(\theta_0)\right\|_w \to 0$$
as $h\to0$. As in $\mathcal L_w$, if $f(\cdot)$ and $\eta(\cdot)$ are differentiable and their derivatives are continuous at $\theta_0$ in $L_w$ and $\mathcal M_w$ respectively, we say that they are continuously differentiable.

Assuming that $(Q(\theta): \theta \in (a,b))$ is $n$-times differentiable in some neighborhood $\mathcal N$ of $\theta_0$, with derivative $(Q^{(n)}(\theta):\theta\in\mathcal N)$, we say that $Q(\cdot)$ is $(n+1)$-times differentiable in $\mathcal L_w$ at $\theta_0$ if $(Q^{(n)}(\theta): \theta\in \mathcal N)$ is differentiable at $\theta_0$, with corresponding derivative $Q^{(n+1)}(\theta_0)$. 
We can analogously define $f^{(n+1)}(\theta_0)$ and $\eta^{(n+1)}(\theta_0)$ in the spaces $L_w$ and $\mathcal M_w$, respectively.
(We restrict our discussion in this paper to scalar $\theta$, since the vector case introduces no new mathematical issues.)

We can now state our first result, pertaining to differentiability of $u^*$. 
\begin{theorem}\label{thm:01} 
Suppose there exists $w:C\to[1,\infty)$ and $\theta_0 \in (a,b)$ for which:
\begin{enumerate}
\item[(a)] $\vertiii{K^m(\theta_0)}_w < 1$ for some $m\geq 1$;

\item[(b)] $K(\cdot)$ is \what{(continuously)} differentiable in $\mathcal L_w$ at $\theta_0$, with derivative $K'(\theta_0)$; 

\item[(c)] $\tilde f(\cdot)$ is \what{(continuously)} differentiable in $L_w$ at $\theta_0$, with derivative $\tilde f'(\theta_0)$. 
\end{enumerate}
Then:
\begin{enumerate}
\item[(i)] $(I-K(\theta))$ is invertible on $L_w$ for $\theta$ in a neighborhood of $\theta_0$;

\item[(ii)] Setting $G(\theta) = (I-K(\theta))^{-1}$, $G(\cdot)$ is \what{(continuously)} differentiable in $\mathcal L_w$ at $\theta_0$, and 
\begin{equation*}
G'(\theta_0) = G(\theta_0) K'(\theta_0) G(\theta_0);
\end{equation*}

\item[(iii)] $u^*(\theta) = \sum_{n=0}^\infty K^n(\theta) \tilde f(\theta)$ is \what{(continuously)} differentiable in $L_w$ at $\theta_0$, with 
\begin{equation}\label{eq:4.3a}
(u^*)'(\theta_0) = G'(\theta_0) \tilde f(\theta_0) + G(\theta_0)\tilde f'(\theta_0).
\end{equation}
\end{enumerate}
If, in addition, $K(\cdot)$ and $\tilde f(\cdot)$ are $n$-times \what{(continuously)} differentiable in $\mathcal L_w$ and $L_w$, respectively, at $\theta_0$, then $Q(\cdot)$ and $u^*(\cdot)$ are $n$-times \what{(continuously)} differentiable at $\theta_0$ in $\mathcal L_w$ and $L_w$, respectively, and $Q^{(n)}(\theta_0)$ and $(u^*)^{(n)}(\theta_0)$ can be recursively computed via
	\begin{align}\label{eq2:D}
	G^{(n)}(\theta_0) = \sum_{j=0}^{n-1} {n\choose j}G^{(j)}(\theta_0)K^{(n-j)}(\theta_0) G(\theta_0)
	\end{align} 
and
	\begin{align}\label{eq2:E}
	(u^*)^{(n)}(\theta_0) = G(\theta_0) \bigg(\tilde f^{(n)}(\theta_0) + \sum_{j=0}^{n-1}{n\choose j} K^{(n-j)} (\theta_0) (u^*)^{(j)}(\theta_0)\bigg)
	\end{align}
where, as usual, $K^{(0)}(\theta)\equiv K(\theta)$ and $\tilde f^{(0)}(\theta) = \tilde f(\theta)$.
\end{theorem}
\begin{proof}{Proof.}
Part (i) is \rvout{well known: see for example,  \what{[reference]}}\rvin{obvious}. For part (ii), note that assumptions (a) and (b) imply that there exists a neighborhood $\mathcal N$ of $\theta_0$ for which $\sup_{\theta\in \mathcal N}\vertiii{K^m(\theta)}_w < 1$ and $\sup_{\theta\in \mathcal N} \vertiii{K(\theta)}_w < \infty$, from which it follows that $\sup_{\theta\in \mathcal N}\vertiii{G(\theta)}_w < \infty$. 
Furthermore, since $(I-K(\theta_0+h)) G(\theta_0+h) = G(\theta_0+h) (I-K(\theta_0+h)) = I$, evidently 
	\begin{equation*}
	(G(\theta_0+h) - G(\theta_0)) (I-K(\theta_0)) = G(\theta_0+h) (K(\theta_0+h) - K(\theta_0)),
	\end{equation*}
so that
	\begin{equation}\label{eq2:2.9}
	G(\theta_0 + h) - G(\theta_0) = G(\theta_0 + h) (K(\theta_0+h) - K(\theta_0)) G(\theta_0)
	\end{equation}
Clearly, this implies that $\vertiii{G(\theta_0+h) - G(\theta_0)}_w \leq \vertiii{G(\theta_0+h)}_w\vertiii{K(\theta_0+h)- K(\theta_0)}_w\vertiii{G(\theta_0)}_w \to 0$ as $h \to 0$, so $G(\cdot)$ is continuous in $\mathcal L_w$ at $\theta_0$. Consequently, (\ref{eq2:2.9}) implies that $G(\cdot)$ is differentiable in $\mathcal L_w$ at $\theta_0$, with $G'(\theta_0) = G(\theta_0) K'(\theta_0)G(\theta_0)$. \what{In case $K'$ is continuous, continuity of $G'$ is also immediate from this expression.}

For part (iii), the result follows analogously from the identity
	\begin{equation*}
	u^*(\theta_0+h) - u^*(\theta_0) = G(\theta_0) (\tilde f(\theta_0+h) - \tilde f(\theta_0)) + (G(\theta_0+h) - G(\theta_0)) \tilde f(\theta_0+h).
	\end{equation*}
The proof for the $n$-fold derivatives for $n\geq 2$ is very similar and therefore omitted.
\end{proof}

\begin{remark}
Suppose that $K(\cdot)$ posesses a density $(k(\cdot, x, y): x,y \in C)$ that is $n$-times differentiable (with (pointwise) derivative $(k^{(n)}(\cdot, x,y):x,y\in C))$. For $\epsilon>0$ and $0\leq j \leq n$, let $\tilde \omega_\epsilon^{(j)}(x,y) = \sup_{|\theta-\theta_0|<\epsilon} |k^{(j)}(x, y)|$. Then, the conditions 
	\begin{align}
	&\sup_{x\in C} \int_C K^m(\theta_0, x, dy) \frac{w(y)}{w(x)} < 1\quad \text{for some } m\geq 1,\label{eq:02.10}\\
	&\sup_{x\in C} \int_C \omega_\epsilon^{(j)}(x,y)\frac{w(y)}{w(x)} K(\theta_0, x, dy) < \infty,\label{eq:02.11}\\
	&\sup_{x\in C} \int_{C^c} (1+\tilde \omega_\epsilon^{(j)}(x,y))\frac{|f(y)|}{w(x)}K(\theta_0, x, dy)  < \infty,\label{eq:02.12}
	\end{align}
for $j=0,\ldots,n$ imply (a), (b), and (c) of Theorem~\ref{thm:01} implying the validity of (\ref{eq2:D}) and (\ref{eq2:E}).
\end{remark}


There is an analgous differentiability results for measures. For a given initial distribution $\mu$ on $C$, let $\nu$ be the measure defined by 
\begin{equation*}
\nu(dy) = \E_\mu \sum_{j=0}^{T-1} \exp \left(\sum_{k=0}^{j-1} g(X_k)\right) \I(X_j \in dy)
\end{equation*}
for $y\in S$, where $\E_\mu(\cdot) \triangleq \int_C \mu(dx) \E_x(\cdot)$. Then, 
\begin{equation*}
\nu = \sum_{n=0}^\infty \mu K^n,
\end{equation*}
where $K$ is defined as in (\ref{eq2:2.3}). Assume that $\mu(\cdot)$ and $K(\cdot)$ now depend on the parameter $\theta$ (so that $\nu$ does as well). The following result has a proof identical to that of Theorem~\ref{thm:01}, and is therefore omitted. 

\begin{theorem}\label{thm:02}
Suppose there exists $w:C\to[1,\infty)$ and $\theta_0 \in (a,b)$ for which:
\begin{itemize}
	\item[(a)] $\vertiii{K^m(\theta_0)}_w < 1$ for some $m\geq 1$; 
	\item[(b)] $K(\cdot)$ is (continuously) differentiable in $\mathcal L_w$ at $\theta_0$, with derivative $K'(\theta_0)$
	\item[(c)] $\mu(\cdot)$ is (continuously) differentiable in $\mathcal M_w$ at $\theta_0$, with derivative $\mu'(\theta_0)$. 
\end{itemize}
Then, $\nu(\theta) = \sum_{n=0}^\infty \mu(\theta)K^m(\theta)$ is (continuously) differentiable in $\mathcal M_w$ in $\theta_0$, with 
\begin{equation*}
\nu'(\theta_0) = \mu'(\theta_0)G(\theta_0) + \nu(\theta_0) G'(\theta_0).
\end{equation*}
If, in addition, $K(\cdot)$ and $\mu(\cdot)$ are $n$-times (continuously) differentiable in $\mathcal L_w$ and $\mathcal M_w$, respectively, at $\theta_0$, then $\nu(\cdot)$ is $n$-times (continuously) differentiable in $\mathcal M_w$, and $\nu^{(n)}(\theta_0)$ can be recursively computed via 
\begin{equation*}
\nu^{(n)} (\theta_0) = \left(\mu^{(n)}(\theta_0) + \sum_{j=0}^{n-1}{n \choose j}\nu^{(j)}K^{(n-j)}(\theta_0)\right)G(\theta_0).
\end{equation*}

\end{theorem}

We finish this section with a short operator-theoretic argument establishing existence of a derivative for the stationary distribution under the assumption of geometric ergodicity (see condition (a) below, which is the key Lyapunov condition that implies geometric ergodicity in Chapter 15 of \cite{meyn09}).

\begin{theorem}\label{thm:4}
Suppose that there exists a subset $A\subseteq S$, $\epsilon,c > 0$, $\lambda, r \in (0,1)$, an integer $m\geq 1$, a probability measure $\varphi$ on $S$, and $w:S\to[1,\infty)$ such that:

\begin{enumerate}
\item[(a)] $(P(\theta_0) w)(x) \leq rw(x)  + cI(x\in A)$ \quad for $x\in S$; 
\item[(b)] $P^m(\theta,x,dy) \geq \lambda \varphi(dy)$ for $x\in A$, $y\in S$, and $|\theta-\theta_0|<\epsilon$;
\item[(c)] $P(\cdot)$ is \what{(continuously)} differentiable in $\mathcal L_w$ at $\theta_0$.
\end{enumerate}
Then, $X$ is positive Harris recurrent for $\theta$ in a neighborhood of $\theta_0$, and the stationary distributions $\pi(\theta)\in \mathcal M_w$ for $\theta$ in a neighborhood of $\theta_0$ are \what{(continuously)} differentiable in $\mathcal M_w$ at $\theta_0$. 
Furthermore, if $\Pi(\theta_0)$ is the kernel defined by $\Pi(\theta_0, x, dy) = \pi(\theta_0, dy)$ for $x,y\in S$, $(I-P(\theta_0)+ \Pi(\theta_0))$ has an inverse on $\mathcal L_w$ and 
	\begin{equation}\label{eq2:2.13}
	\pi'(\theta_0) = \pi(\theta_0) P'(\theta_0) (I-P(\theta_0)+\Pi(\theta_0))^{-1}.
	\end{equation}
If, in addition, $P(\cdot)$ is $n$-times \what{(continuously)} differentiable in $\mathcal L_w$ at $\theta_0$, then $\pi(\cdot)$ is $n$-times \what{(continuously)} differentiable in $\mathcal M_w$ at $\theta_0$, and $\pi^{(n)}(\theta_0)$ can be recursively computed via 
	\begin{equation*}
	\pi^{(n)}(\theta_0) = \sum_{j=0}^{n-1} {n \choose j} \pi^{(j)}(\theta_0) P^{(n-j)}(\theta_0) (I-P(\theta_0) + \Pi(\theta_0))^{-1}.
	\end{equation*}
\end{theorem}
\begin{remark}
 Note that Theorem 4 of \cite{glynn95} is closely related to the above theorem. See also Remark 11 and the Kendall set assumption in \cite{glynn95}. \cite{heidergott06} also imposes similar assumption to establish the measure-valued derivative of the stationary distribution. 
\end{remark}
\begin{proof}{Proof.}
In view of (a) and (c), there exists $r'<1$ such that 
	\begin{equation}\label{eq2:2.14}
	(P(\theta_0+h) w)(x) \leq r' w(x) + cI(x\in A)
	\end{equation}
for $x\in S$ and $|h|$ sufficiently small. Assumptions (a) and (b), and the fact that $w\geq 1$ implies that $X$ is positive Harris recurrent for $\theta$ in a neighborhood of $\theta_0$. We can now appeal to Theorem 2.3 of \cite{glynn96} to establish that $(I-P(\theta_0) + \Pi(\theta_0))$ is invertible on $\mathcal L_w$, with $(I-P(\theta_0) + \Pi(\theta_0))^{-1} \in \mathcal L_w$. 

Furthermore, according to \cite{glynn08}, (\ref{eq2:2.14}) implies that $\pi(\theta_0+h) w \leq c/(1-r')$, and hence $\|\pi(\theta_0+h)\|_w \leq c/(1-r').$ Also,
	\begin{align*}
	(\pi(\theta_0+h) - \pi(\theta_0)) (I-P(\theta_0)) 
	&= \pi(\theta_0+h) (I-P(\theta_0))\\
	&= \pi(\theta_0+h) (P(\theta_0+h) - P(\theta_0)).
	\end{align*}
In addition, $\nu \Pi(\theta_0) = \pi(\theta_0) $ for any probability $\nu$ on $S$. So $(\pi(\theta_0+h) - \pi(\theta_0))\Pi(\theta_0) = 0$. Consequently, 
	\begin{equation*} 
	(\pi(\theta_0 + h) - \pi(\theta_0)) (I-P(\theta_0)+\Pi(\theta_0)) 
	= \pi(\theta_0+h) (P(\theta_0+h)-P(\theta_0)),
	\end{equation*}
from which it follows that
	\begin{equation}\label{eq2:2.15}
	\pi(\theta_0 + h) - \pi(\theta_0) 
	= \pi(\theta_0+h) (P(\theta_0+h)-P(\theta_0))(I-P(\theta_0)+\Pi(\theta_0))^{-1}.
	\end{equation}
Thus,
	\begin{equation}\label{continuity_of_pi}
	\|\pi(\theta_0+h) - \pi(\theta_0)\|_w 
	\leq \frac{c}{1-r'}\vertiii{P(\theta_0+h) - P(\theta_0)}_w\cdot \vertiii{(I-P(\theta_0)+\Pi(\theta_0))^{-1}}_w
	\end{equation}
Since $P(\cdot)$ is differentiable in $\mathcal L_w$, $\vertiii{P(\theta_0+h)-P(\theta_0)}_w \to 0$ as $h \to 0$, so $\pi(\theta_0+h) \to \pi(\theta_0)$ in $\mathcal M_w$ as $h\to 0$. Letting $h\to 0$ in $(\ref{eq2:2.15})$ then yields (\ref{eq2:2.13}).

\what{
For the continuity of the derivative in case $P(\cdot)$ is continuously differentiable, note first that (a) and (b)
imply that $\vertiii{(P(\theta_0)-\Pi(\theta_0))^m}_w<1$ for some $m\geq 1$; this along with the continuity of $P(\cdot)$ and $\pi(\cdot)$, in turn, implies that $\sup_{|h|\leq h_0}\vertiii{(P(\theta_0+h)-\Pi(\theta_0+h))^m}_w<1$ for a small enough $h_0$. Therefore, we conclude that $\vertiii{(I-P(\theta_0+h) + \Pi(\theta_0+h))^{-1}}_w$ is bounded (uniformly w.r.t.\ $h$). 
From this, it is easy to see that the same argument as for (\ref{eq2:2.13}) works with $\theta = \theta_0+h$ instead of $\theta_0$ and proves that
\begin{equation}\label{pi_prime_general}
\pi'(\theta_0+h) = \pi(\theta_0+h) P'(\theta_0+h) (I-P(\theta_0+h)+\Pi(\theta_0+h))^{-1}.	
\end{equation}
Now, 
\begin{align*}
\pi'(\theta_0+h) - \pi'(\theta_0)
&= 
\left(\pi'(\theta_0+h) - \frac{\pi(\theta_0)-\pi(\theta_0+h)}{-h}\right) - \left(\pi'(\theta_0)- \frac{\pi(\theta_0+h)-\pi(\theta_0)}{h} \right)
= \text{(I)} - \text{(II)}
\end{align*}
where we have already seen that (II) converges to 0.
To show that (I) also vanishes, note that
\begin{align*}
(\pi(\theta_0) - \pi(\theta_0+h))(I-P(\theta_0+h)+\Pi(\theta_0+h))  
&
=(\pi(\theta_0) - \pi(\theta_0+h))(I-P(\theta_0+h)) 
\\
&
= \pi(\theta_0)(I-P(\theta_0+h))
= \pi(\theta_0)(P(\theta_0)-P(\theta_0+h)),
\end{align*}
and hence,
\begin{equation}\label{pi_finite_difference}
\frac{\pi(\theta_0) - \pi(\theta_0+h)}{-h}
= 
\pi(\theta_0)\frac{P(\theta_0+h)-P(\theta_0)}{h}(I-P(\theta_0+h)+\Pi(\theta_0+h))^{-1}.
\end{equation}
From (\ref{pi_prime_general}), (\ref{pi_finite_difference}), the continuity of $\pi(\cdot)$, the continuous differentiability of $P(\cdot)$, and the uniform boundedness of the norm of $(I-P(\theta_0+h)+\Pi(\theta_0+h))^{-1}$, we conclude that (I) vanishes. Therefore, $\pi'(\cdot)$ is continuous at $\theta_0$.
}

Finally, as in Proposition 2, the proof for the $n$-fold derivatives for $n\geq 2$ follows similar lines, and is therefore omitted. 
\end{proof}

This result establishes, in the presence of a single Lyapunov function $w$, the $n$-fold differentiability of the stationary distribution $\pi(\cdot)$ in $\mathcal M_w$. Of course, the simplicity of the result comes at the cost of assuming geometric ergodicity of $X$.

%
%
\section{Lyapunov Criteria for Differentiability of Random Horizon Expectations}\label{sec:random_horizon}

Let $\Lambda = (a,b)$ be an open interval containing $\theta_0$. For each $\theta \in \Lambda$, let $\E_x^\theta (\cdot) \triangleq \E^\theta(\cdot|X_0 = x)$ be the expectation operator associated with $X$, when $X$ is driven by the one-step transition kernel $P(\theta)$. 
As in Section~\ref{sec:operator}, we consider 
\begin{align}\label{def:u_star}
u^*(\theta, x) = \E_x^\theta \sum_{j=0}^{T-1} \exp\left(\sum_{k=0}^{j-1} g(X_k)\right)f(X_j) + \exp\left(\sum_{k=0}^{T-1} g(X_k)\right)f(X_T)
\end{align}
for each $x\in C$ 
given $f:S\to \R$, $g:S\to \R$, $\phi \neq C\subseteq S$, and $T=\inf \{n\geq 0: X_n \in C^c\}$. 
Our goal, in this section, is to provide Lyapunov conditions under which $u^*(\theta) = (u^*(\theta,x): x \in C)$ is differentiable at $\theta_0$, and to provide an expression for the derivative ${u^*}'(\theta)$.


Note that if $f$ is non-negative, then $u^*(\theta)$ is always well-defined. Furthermore, by conditioning on $X_1$, it is easily seen that 
\begin{align*}
u^*(\theta, x) = f(x) + \int_{C^c} \exp(g(x)) P(\theta,x,dy)f(y) + \int_C \exp(g(x))P(\theta, x,dy) u^*(\theta,y)
\end{align*}
for $x\in C$, and hence
\begin{equation}\label{eq:u_star}
u^*(\theta) = \tilde f(\theta) + K(\theta)u^*(\theta),
\end{equation}
where as in Section~\ref{sec:operator}
\begin{equation*}
\tilde f(\theta,x) = f(x) + \int_{C^c} \exp(g(x)) P(\theta,x,dy)f(y)
\end{equation*}
for $x\in C$, and $K(\theta) = (K(\theta, x, dy): x, y \in C)$ is the non-negative kernel on $C$ for which 
\begin{equation*}
K(\theta, x, dy ) = \exp(g(x))P(\theta,x,dy).
\end{equation*}






Given (\ref{eq:u_star}), formal differentiation of both sides of the equation yields
\begin{equation}
{u^*}'(\theta_0) = \tilde f'(\theta_0) + K'(\theta_0){u^*}(\theta_0) + K(\theta_0) {u^*}'(\theta_0),
\end{equation}
so that ${u^*}'(\theta)$ should satisfy the linear system
\begin{equation}\label{eq:2.3a}
(I-K(\theta_0)){u^*}'(\theta_0) = \tilde f'(\theta_0) + K'(\theta_0)u^*(\theta_0).
\end{equation}
When $|C|$ is finite, it will frequently be the case that the matrix $K(\theta_0)$ has spectral radius less than $1$, in which case $I-K(\theta_0)$ is invertible and 
\begin{equation}\label{eq:potential}
(I-K(\theta_0))^{-1} = \sum_{n=0}^\infty K^n(\theta_0)
\end{equation}
In this case, 
\begin{equation*}
{u^*}'(\theta_0) = \sum_{n=0}^\infty K^n(\theta_0)\left(\tilde f'(\theta_0) + K'(\theta_0) u^*(\theta_0)\right).
\end{equation*}
But (\ref{eq:u_star}) and (\ref{eq:potential}) further imply that
\begin{equation}\label{rep:u_star}
u^*(\theta_0) = \sum_{n=0}^\infty K^n(\theta_0)\tilde f(\theta_0), 
\end{equation}
and hence we arrive at the formula
\begin{equation}\label{eq:2.5}
{u^*}'(\theta_0) = \sum_{m=0}^\infty\sum_{n=0}^\infty K^m(\theta_0) K'(\theta_0) K^n(\theta_0) \tilde f(\theta_0) + \sum_{m=0}^\infty K^m(\theta_0) \tilde f'(\theta_0).
\end{equation}
The remainder of this section is largely concerned with rigorously extending the formula (\ref{eq:2.5}) to the general state space setting, under Lyapunov criteria that are close to minimal (and easily checkable from the model building blocks). We start by observing that when $f$ is non-negative, Fubini's theorem implies that 
\begin{align}
u^*(\theta, x) 
&= \sum_{j=0}^\infty \E_x^\theta \exp\left(\sum_{k=0}^{j-1} g(X_k)\right)f(X_j)I(T>j)\nonumber\\
&\qquad + \sum_{j=0}^\infty \E_x^{\theta} \exp\left( \sum_{k=0}^{j-1} g(X_k)\right) I(T\geq j) \hcancel{e^{g(X_j)}}f(X_j) \I(X_j \in C^c)\nonumber\\
&= \sum_{j=0}^\infty (K^{j}(\theta)\tilde f(\theta))(x),\label{eq:2.5a}
\end{align}
thereby rigorously verifying (\ref{rep:u_star}). To simplify the notation in the remainder of this paper, we set $K = K(\theta_0)$ and put 
\begin{equation}\label{eq:2.6}
G = \sum_{n=0}^\infty K^n.
\end{equation}

Our path to providing rigorous conditions under which (\ref{eq:2.5}) holds involves the following key ``absolute continuity'' assumption:
\begin{itemize}
	\item[A1.] 
	The kernels $(K(\theta): \theta \in \Lambda)$ are absolutely continuous with respect to $K$, in the sense that there exists a (measurable) density $(k(\theta, x,y): x,y\in C)$ such that 
\begin{equation*}
K(\theta, x,dy) = k(\theta, x, y) K(x,dy)
\end{equation*}
for $\theta\in \Lambda$, $x,y\in C$.
\end{itemize}
Our absolute continuity condition is often a mild hypothesis. For example, when $X$ has a transition density with respect to a reference measure $\eta$, A1 is in force when the support of the density is independent of $\theta$.

We also need to assume that $K(\theta)$ is suitably differentiable at $\theta_0$.
\begin{itemize}
	\item[A2.] 
	There exists $\epsilon>0$ such that for each $x,y\in C$, $k(\cdot, x,y)$ is continuously differentiable, with derivative $k'(\cdot, x,y)$, in $[\theta_0-\epsilon, \theta_0+\epsilon]$.
\end{itemize}

Set $\omega_\epsilon(x,y) = \sup\{|k'(\theta,x,y)|: |\theta-\theta_0|<\epsilon\}$, $k'(x,y) = k'(\theta_0, x,y)$, and $K'(x,dy) = k'(x,y) \allowbreak  K(x,dy)$. (Note that $K'$ is a signed kernel, and not non-negative.)

Our hypotheses are stated in terms of $K(\theta)$, not $P(\theta)$, in order to offer the extra generality needed to cover settings in which derivatives involving parameters in the discount factor $\exp(g(\cdot))$ are of interest. Such derivatives are commonly considered in the finance literature when attempting to hedge uncertainty in the so-called ``short rate.'' (The resulting derivative is called \emph{rho} in the finance context.)

Finally, we also need to assume $\tilde f(\theta)$ is suitably differentiable at $\theta_0$. To permit derivatives in parameters that involve the discount factor, we write $\tilde f(\theta)$ in the form
\begin{equation}
\tilde f(\theta, x) = f(x) + \int_{C^c} K(\theta, x, dy) f(y).
\end{equation}
\begin{itemize}
	\item[A3.] 
	The family of measures $(K(\theta, x, dy): \theta\in \Lambda, x \in C, y \in C^c)$ is absolutely continuous with respect to $(K(\theta_0, x, dy): x\in C, y \in C^c)$, in the sense that there exists a (measurable) density $(k(\theta, x,y): x\in C, y \in C^c)$ such that 
	\begin{equation*}
	K(\theta, x, dy) = k(\theta, x, y) K(\theta_0, x, dy)
	\end{equation*}
	for $\theta \in \Lambda$, $x \in C$, $y\in C^c$. Furthermore, there exists $\epsilon>0$ such that for $x \in C$, $y\in C^c$, $k(\cdot, x,y)$ is continuously differentiable, with derivative $k'(x,y),$ in $[\theta_0-\epsilon, \theta_0 +\epsilon]$. Also, we assume that 
	\begin{equation*}
	\tilde r_\epsilon(x) \triangleq \int_{C^c} \tilde \omega_\epsilon(x,y)|f(y)|K(\theta_0,x,dy)<\infty
	\end{equation*}
	for $x \in C$, where
	\begin{equation*}
	\tilde \omega_\epsilon(x,y) = \sup_{|\theta - \theta_0 | < \epsilon} |k'(\theta, x,y)|.
	\end{equation*}
\end{itemize}
In many applications, $\tilde f(\theta)$ is independent of $\theta$ and A3 need not be verified (e.g. expected hitting times).

For $x\in C$, $y\in C^c$, set $K(x,dy)= K(\theta_0, x, dy)$ and $K'(x,dy)= k'(\theta_0, x,y) K(x,dy)$. We are now ready to state the main theorem of this section.

\begin{theorem}\label{thm:1}
Assume A1, A2, and A3. Suppose there exists $\epsilon > 0$ and two finite-valued non-negative functions $v_0$ and $v_1$ defined on $C$ for which
\begin{equation}\label{eq:2.7}
(K(\theta)v_0)(x) \leq v_0(x) - |\tilde f(\theta, x)|
\end{equation}
for $x\in C$ and $|\theta - \theta_0| < \epsilon$, and 
\begin{equation}\label{eq:2.8}
(Kv_1)(x) \leq v_1(x) - \int_C \omega_\epsilon(x,y) v_0(y) K(x,dy) - \tilde r_\epsilon(x)
\end{equation}
for $x\in C$. Then, $u^*(\cdot,x)$ is differentiable at $\theta_0$ and 
\begin{equation}\label{eq:2.9}
{u^*}'(\theta_0) = \int_C\int_C\int_C G(x,dy) K'(y,dz) G(z, dw) f(w) + \int_C \int_{C^c} G(x, dy) K'(y,dz) f(z).
\end{equation}
\what{%
If, in addition, 
\begin{equation}\label{cond:random-horizon-C1-2}
\int_C \omega_\epsilon(x,y)v_1(y)K(x,dy)<\infty
\end{equation}
and (\ref{eq:2.8}) holds in a neighborhood of $\theta_0$, i.e., for $\theta\in[\theta_0-\epsilon, \theta_0+\epsilon]$
\begin{equation}\label{cond:random-horizon-C1-3}\tag{\ref*{eq:2.8}$'$}
(K(\theta)v_1)(x) \leq v_1(x) - \int_C \omega_\epsilon(x,y) v_0(y) K(\theta,x,dy) - \int_{C^c} \tilde \omega_\epsilon(x,y)|f(y)|K(\theta,x,dy)
\end{equation}
then ${u^*}'(\cdot,x)$ is continuous on $[\theta_0-\epsilon, \theta_0+\epsilon]$. 
}
\end{theorem}
Recalling the definition of $G$, we see that (\ref{eq:2.9}) is indeed the general state space analog of (\ref{eq:2.5}). The functions $v_0$ and $v_1$ appearing in Theorem~\ref{thm:1} are often called (stochastic) Lyapunov functions. A standard means of guessing good choices for $v_0$ and $v_1$ is to recognize that $u^*(\theta_0)$ satisfies (\ref{eq:2.7}) with equality, 
\rvin{if $\tilde f$ is non-negative}
while 
\begin{equation*}
\int_C \left[\int_C K(y, dz) \omega_\epsilon(y,z) v_0(z) + r_\epsilon(y)\right] G(x,dy)
\end{equation*}
satisfies (\ref{eq:2.8}) with equality. When $C\subseteq \R^m$ is unbounded, one can often approximate the large $x$ behavior of these functions, and use these approximations as choices for $v_1$ and $v_2$, respectively.

The proof of Theorem~\ref{thm:1} rests on the following easy bound. 
\begin{proposition}\label{prop:1}
	Suppose that $Q=(Q(x,dy): x,y\in C)$ is a non-negative kernel and that $f: C\to \R_+$. If $v:C\to \R_+$ is a finite-valued function for which 
	\begin{equation}\label{eq:2.10}
	Qv \leq v - f,
	\end{equation}
	then
	\begin{equation}\label{eq:2.11}
	\sum_{n=0}^\infty Q^n f \leq v.
	\end{equation}
\end{proposition}
\begin{proof}{Proof.}
Note that (\ref{eq:2.10}) implies that $Qv \leq v$, and hence $Q^nv \leq v$ for $n\geq 0$. It follows that $Q^nv$ is finite-valued for $n\geq 0$. Inequality (\ref{eq:2.10}) can be re-written as 
\begin{equation}\label{eq:2.12}
f \leq v - Qv.
\end{equation}
Applying $Q^j$ to both sides of (\ref{eq:2.12}), we get
\begin{equation}\label{eq:2.13}
Q^jf\leq Q^j v - Q^{j+1}v
\end{equation}
Summing both sides of (\ref{eq:2.13}) over $j = 0, 1, \ldots, n$, we find that 
\begin{equation*}
\sum_{j=0}^n Q^j f \leq v - Q^{n+1} v \leq v.
\end{equation*}
Sending $n \to \infty$ yields (\ref{eq:2.11}). 
\end{proof}

\begin{proof}{Proof of Theorem~\ref{thm:1}.}
For the purposes of this proof, $\epsilon$ is taken as the smallest of the $\epsilon$'s appearing in A2, A3, and the statement of the theorem. We start by observing that Proposition~\ref{prop:1}, applied to the Lyapunov bound (\ref{eq:2.7}), guarantees that
\begin{equation*}
\sum_{n=0}^\infty K^n(\theta) |\tilde f(\theta) | \leq v_0
\end{equation*}
and hence \rvout{(\ref{rep:u_star})} \rvin{Fubini's theorem} implies that $u^*(\theta)$ is finite-valued, $u^*(\theta) = \sum_{n=0}^\infty K^n(\theta)\tilde f(\theta)$, and $|u^*(\theta)|\leq v_0$. Since $u^*(\theta)$ is finite-valued (as is $K(\theta)u^*(\theta)$), we can write 
\begin{equation*}
u^*(\theta_0+h) - u^*(\theta_0) = K(\theta_0+h) u^*(\theta_0+h) - K(\theta_0) u^*(\theta_0) + \tilde f(\theta_0 + h) - \tilde f(\theta_0)
\end{equation*}
and hence
\begin{equation}\label{eq:2.14}
(I-K) \big(u^*(\theta_0+h)- u^*(\theta_0)\big) = \big(K(\theta_0 + h)- K(\theta_0) \big) u^*(\theta_0 + h) + \big(\tilde f(\theta_0 + h) - \tilde f(\theta_0)\big).
\end{equation}
For $|h| < \epsilon$,
\begin{align*}
&\left| \int_C (K(\theta_0 + h, x, dy) - K(\theta_0, x,dy))u^*(\theta_0+h,y) \right|\\
&\leq \int_C | k(\theta_0 + h, x, y) - k(\theta_0 , x , y)| K(x,dy) v_0(y)\\
&\leq |h| \int_C \sup_{|\theta - \theta_0| < \epsilon} |k'(\theta, x,y) | K(x,dy) v_0(y)\\
&= |h| \int_C \omega_\epsilon(x,y) K(x,dy) v_0(y).
\end{align*}
Similarly, for $|h|<\epsilon$,
\begin{align*}
&|\tilde f(\theta_0 + h, x) - \tilde f(\theta_0, x)|\\
&\leq |h| \int_{C^c} \tilde \omega_\epsilon (x,y)K(x,dy)|f(y)|\\
&\leq |h| \tilde r_\epsilon (x).
\end{align*}
Consequently, Proposition~\ref{prop:1}, together with the Lyapunov bound (\ref{eq:2.8}), ensures that 
\begin{equation*}
\int_C G(x,dy)\bigg( \left| \int_C (K(\theta_0 + h, y, dz) - K(\theta_0, y, dz) ) u^*(\theta_0+h, z) \right| +\left|\tilde f(\theta_0+h,y) - \tilde f(\theta_0, y)\right|\bigg) \leq |h|v_1(x).
\end{equation*}
It follows from (\ref{eq:2.14}) that $u^*(\theta,x)$ is continuous at $\theta_0$ and
\cf{
Is there a short justification? Referees may ask to justify.
}
\begin{align*}
\frac{u^*(\theta_0+h,x) - u^*(\theta_0,x)}{h} 
&= \int_C G(x,dy) \left[\int_C \frac{k(\theta_0+h,y,z) - k(\theta_0, y,z)}{h} u^*(\theta_0+h,z)K(y,dz) \right.\\
&\qquad\qquad\quad\qquad+ \left.\int_{C^c} \frac{k(\theta_0 + h, y,z) - k(\theta_0, y,z)}{h} f(z) K(y,dz)
\right].
\end{align*}
But 
\begin{equation}\label{eq:2.15}
\frac{k(\theta_0+h,y,z)-k(\theta_0, y, z)}{h} \to k'(y,z)
\end{equation}
and 
\begin{equation}\label{eq:2.16}
u^*(\theta_0+h,z) \to u^*(\theta_0, z)
\end{equation}
as $h\to 0$. Also,
\begin{equation}\label{eq:2.17}
\left|\frac{k(\theta_0+h,y,z) - k(\theta_0,y,z)}{h}u^*(\theta_0+h,z)\right| \leq \omega_\epsilon(y,z) v_0(z)
\end{equation}
for $y,z\in C$, and 
\begin{equation}\label{eq:2.18}
\frac{|k(\theta_0+h,y,z) - k(\theta_0,y,z)|}{h} \leq \tilde \omega_\epsilon(y,z)
\end{equation}
for $y\in C$, $z\in C^c$. The Lyapunov bound (\ref{eq:2.8}), together with Proposition~\ref{prop:1}, guarantees that
\begin{equation}\label{eq:2.19}
\int_C G(x,dy)\left(\int_C \omega_\epsilon(y,z) v_0(z) K(y,dz) + \int_{C^c} \tilde \omega_\epsilon (y,z) |f(z)| K(y,dz) \right) < \infty.
\end{equation} 
In view of (\ref{eq:2.15}) through (\ref{eq:2.19}), the Dominated Convergence Theorem therefore establishes that $u^*(\theta, x)$ is differentiable at $\theta_0$, and 
\begin{equation}\label{eq:sec3:u-star-prime-in-the-proof}
{u^*}'(\theta_0, x) = \int_C G(x,dy) \int_C k'(y,z) u^*(\theta_0, z) K(y,dz) + \int_C G(x,dy)\int_{C^c} k'(y,z)f(z) K(y,dz),
\end{equation}
which is equivalent to (\ref{eq:2.9}).

\what{%
Turning to the continuity of ${u^*}'(\cdot, x)$, note that one can easily check that 
$$
{u^*}'(\theta) = \tilde f'(\theta) + K'(\theta)u^*(\theta) + K(\theta){u^*}'(\theta)
$$
for $\theta\in[\theta_0-\epsilon, \theta_0+\epsilon]$ where $K'(\theta)u^*(\theta,x) = \int_C k'(\theta, x,y) u^*(\theta,y)K(x,dy)$, and hence,
\begin{align*}
{u^*}'(\theta+h) - {u^*}'(\theta) 
&= 
G(\theta)
\big(
\tilde f'(\theta+h) - \tilde f'(\theta) 
\big)
+
G(\theta)
\big(
(K'(\theta+h)-K'(\theta))u^*(\theta+h)
\big)\\
&\hspace{20pt}
+
G(\theta)
\big(
K'(\theta)(u^*(\theta+h)-u^*(\theta))
\big)
+
G(\theta)
\big(
(K(\theta+h)-K(\theta)){u^*}'(\theta+h)
\big).
\end{align*}
Now, a similar argument (via dominated convergence and the Lyapunov conditions) as the one that leads to (\ref{eq:sec3:u-star-prime-in-the-proof})---along with (\ref{cond:random-horizon-C1-2}), and (\ref{cond:random-horizon-C1-3})---shows that ${u^*}'(\theta+h) - {u^*}'(\theta) \to 0$ for $\theta\in[\theta_0-\epsilon, \theta_0+\epsilon]$.
}
\end{proof}

Our proof also yields the following (computable) bound on ${u^*}'(\theta_0)$, namely, 
\begin{equation}\label{eq:2.20}
|{u^*}'(\theta_0, x) | \leq v_1(x)
\end{equation}
for $x \in C$. 

In many applications, the parameter $\theta$ enters the dynamics in a very specific way, which allows further simplification of the result. 
In particular, whenever $S$ is a separable metric space, we can always express $X$ as the solution to a stochastic recursion; see, for example, 
\cite{kifer1986}. Namely, we can find a mapping $r:S\times S' \to S$ and a sequence $(Z_n:n\geq 1)$ of independent and identically distributed (iid) $S'$-valued random elements such that 
\begin{equation}\label{eqA}
X_{n+1} = r(X_n, Z_{n+1})
\end{equation}
for $n\geq 0$. Suppose that $\theta$ affects the dynamics of $X$ only through the distribution of the $Z_n$'s. Assume that for $z\in S'$,
\begin{equation}\label{eqAB}
P^\theta(Z_1 \in dz) = p(\theta, z) P^{\theta_0} (Z_1\in dz),
\end{equation}
where $p(\cdot,z)$ is continuously differentiable for $z\in S'$. If $u^*(\theta, x)$ is defined as in (\ref{def:u_star}), then $u^*(\cdot, x)$ is differentiable at $\theta_0$ and ${u^*}'(\theta_0,x)$ is given by (\ref{eq:2.9}) (where $K'(x,dy) = \E^{\theta_0}\I(r(x,Z_1)\in dy)p'(\theta_0, Z_1)$), provided that there exists $\epsilon>0$ and finite-valued non-negative function $v_0$ and $v_1$ defined on $C\subseteq S$ for which
\begin{equation}\label{eqC}
\E^{\theta_0} v_0(r(x, Z_1)) p(\theta, Z_1) \leq v_0(x) - |\tilde f(\theta, x)|
\end{equation}
for $x \in C$ and $|\theta - \theta_0| < \epsilon$, and 
\begin{align*}\label{eqD}
\E^{\theta_0} v_1 (r(x,Z_1)) 
\leq v_1 (x) &- \E ^{\theta_0}v_0(r(x,Z_1)) \sup_{|\theta - \theta_0|<\epsilon} |p'(\theta,Z_1)| \I(r(x,Z_1) \in C) \\
&- \E^{\theta_0} |f(r(x,Z_1)) |\sup_{|\theta - \theta_0|< \epsilon} |p'(\theta,Z_1)| \I(r(x,Z_1)\in C^c).
\end{align*}
for $x\in C$; the proof is essentially identical to that of Theorem~\ref{thm:1} and is omitted.

According to Theorem~\ref{thm:1}, for functions $f$ satisfying the Lyapunov bound, 
\begin{equation*}
{u^*}'(\theta_0, x) = \int_S \nu'(x,dy)f(y)
\end{equation*}
where
\begin{align*}
\nu'(w,dz) = \begin{cases}
\int_C G(w,dx)\int_C K'(x,dy) \int_C G(y,dz), & w,z\in C\\
\int_C G(w,dx)\int_{C^c}K'(x,dz), &  w\in C, z \in C^c.
\end{cases}
\end{align*}
Hence, our derivative can be represented in terms of a signed measure. (In general, $\nu'(x,S)$ is non-zero in this setting.)

The above approach also extends, in a straightforward way, to higher-order derivatives. Formal differentiation of (\ref{eq:u_star}) $n$ times yields the identity
\begin{equation*}
u^{*(n)}(\theta) = \tilde f^{(n)} (\theta) + \sum_{j=0}^n \binom{n}{j} K^{(n-j)}(\theta) u^{*(j)}(\theta),
\end{equation*}
which suggests that the $n$\textsuperscript{th} order derivative $u^{*(n)}(\theta)$ can then be recursively computed from $u^{*(0)}(\theta)$, \ldots, $u^{*(n-1)}(\theta)$ by solving the linear (integral) equation
\begin{equation}
(I-K(\theta))u^{*(n)} (\theta) 
= \rvin{\tilde f^{(n)}(\theta)} + \sum_{j=0}^{n-1}\binom{n}{j} K^{(n-j)} (\theta) u^{*(j)}(\theta).
\end{equation}
In particular, it should follow that 
\begin{equation}\label{eq:2.21}
u^{*(n)}(\theta) = G\left(\tilde f^{(n)}(\theta)+\sum_{j=0}^{n-1}\binom{n}{j} K^{(n-j)}(\theta) u^{*(j)}(\theta)\right).
\end{equation}

Rigorous verification of (\ref{eq:2.21}) can be implemented with a family $v_0, v_1, \ldots, v_n$ of Lyapunov functions. Specifically assume that the densities $k(\cdot, x, y)$ (for $x\in S, y \in S$) are $n$-times continuously differentiable in some neighborhood $[\theta_0-\epsilon, \theta_0 + \epsilon]$ of $\theta_0$, and set 
\begin{equation*}
\omega_\epsilon^{(j)}(x,y) = \sup_{|\theta-\theta_0|<\epsilon} |k^{(j)}(\theta, x, y)|
\end{equation*}
for $x,y\in C$ and 
\begin{equation*}
\tilde \omega_\epsilon^{(j)}(x,y) = \sup_{|\theta-\theta_0|<\epsilon} |k^{(j)}(\theta, x, y)|
\end{equation*}
for $x\in C$, $y\in C^c$.

\begin{theorem}\label{thm:2}
	Suppose that there exists $\epsilon > 0$ and a family of finite-valued non-negative functions $v_0, v_1, \ldots, v_n$ defined on $C$ for which 
	\begin{equation*}
	(K(\theta)v_0)(x) \leq v_0(x) - |\tilde f(\theta,x)|
	\end{equation*}
	for $x\in C$ and $|\theta - \theta_0| < \epsilon$;
	\begin{align*}
	(K(\theta)v_l)(x) \leq v_l(x) - \sum_{j=0}^{l-1} \binom{l}{j} \int_C \omega_\epsilon^{(l-j)} (x,y) v_j(y)K(\theta,x,dy) - \int_{C^c} \tilde \omega_\epsilon^{(l)}(x,y) |f(y)| K(\theta,x,dy)
	\end{align*}
	for $x\in C$, $|\theta - \theta_0| < \epsilon$, and $1\leq l\leq n$; \what{and
	$$
	\int_C \omega_\epsilon^{(n)}v_n(y)K(x,dy) < \infty
	$$
	for $x\in C.$
	}
	Then, $u^*(\cdot, x)$ is $n$-times continuously differentiable at $\theta_0$, and the derivative can be recursively computed from the equations 
	\begin{align*}
	u^{*(l)}(\theta_0,x) 
	&= \int_C G(x,dy) 
	\int_C\sum_{j=0}^{l-1}\binom{l}{j} k^{(l-j)} (\theta_0, x, y) u^{*(j)}(y)K(x,dy)\\
	&\qquad\quad+\int_C G(x,dy)\int_{C^c} k^{(l)}(y,z)f(z) K(y,dz)
	\end{align*}
\end{theorem}
The proof of Theorem~\ref{thm:2} mirrors that of Theorem~\ref{thm:1}, and is therefore omitted. As in the proof of Theorem~\ref{thm:1}, the argument establishes the bound $|u^{*(n)}(\theta_0, x)| \leq v_n(x)$ for $x\in C$ on the $n$\textsuperscript{th} order derivative.

%
%
\section{Lyapunov Criteria for Differentiability of Stationary Expectations}\label{sec:equilibrium}

Perhaps the most commonly occurring expectations that arise in applications are those associated with steady-state behavior. Our Lyapunov approach is also well-suited to establishing differentiability in this context. As in Section~\ref{sec:random_horizon}, it is informative to first study the problem non-rigorously.

A stationary distribution $\pi(\theta) = (\pi(\theta, dx): x \in S)$ of the Markov chain $X$ associated with one-step transition kernel $P(\theta)$ will satisfy
\begin{equation}\label{eq:3.1a}
\pi(\theta) = \pi(\theta)P(\theta).
\end{equation}
Differentiating both sides of (\ref{eq:3.1a}) with respect to $\theta$, we obtain 
\begin{equation*}
\pi'(\theta) = \pi'(\theta)P(\theta) + \pi(\theta)P'(\theta),
\end{equation*}
which leads to the equation 
\begin{equation*}
\pi'(\theta) (I-P(\theta))=\pi(\theta)P'(\theta).
\end{equation*}
This equation is similar to (\ref{eq:2.3a}). However, unlike (\ref{eq:2.3a}), the operator $I-P(\theta)$ appearing here will never be invertible, even when $|S|<\infty$. In addition, $I - P(\theta)$ is acting on a measure rather than a function in this setting. Thus, a different approach is needed here. 

For a given function $f: S\to \R$, set $\alpha(\theta) = \pi(\theta) f$. Thus,
\begin{align}
\alpha(\theta_0 + h) - \alpha (\theta_0) 
&= \pi(\theta_0+h)f-\pi(\theta_0)f\nonumber\\
&= \pi(\theta_0+h)f_c,\label{eq:3.1}
\end{align}
where $f_c(x) = f(x) - \pi(\theta_0) f$. While $I-P(\theta_0)$ is singular, the \emph{Poisson's equation} 
\begin{equation}\label{eq:3.2}
(I-P(\theta_0))g = f_c
\end{equation}
is, under suitable technical conditions, generally solvable for $g$ (because of the special structure of the right-hand side, namely $\pi(\theta_0)f_c= 0$). Substituting (\ref{eq:3.2}) into (\ref{eq:3.1}), we get 
\begin{align}
\alpha(\theta_0 + h) - \alpha(\theta_0)
&= \pi(\theta_0 + h) (I-P(\theta_0))g\nonumber\\
&= \pi(\theta_0 + h) (P(\theta_0+h) - P(\theta_0))g.\label{eq:3.3a}
\end{align}
This suggests that
\begin{equation}\label{eq:3.3}
\alpha'(\theta_0) = \pi(\theta_0) P'(\theta_0)g.
\end{equation}
We now turn to making this argument rigorous.

We start by assuming that $(P(\theta): \theta \in \Lambda)$ itself satisfies the absolute continuity condition:
\begin{itemize}
	\item[A4.] 
	The family of one-step transition kernels $(P(\theta): \theta \in \Lambda)$ is absolutely continuous with respect to $P(\theta_0)$, in the sense that there exists a density $(p(\theta, x, y): \theta\in \Lambda, x,y \in S)$ for which 
	\begin{equation*}
	P(\theta, x, dy) = p(\theta, x, y) P(\theta_0, x, dy)
	\end{equation*}
	for $x, y \in S$, and $\theta\in \Lambda$. Furthermore, there exists $\epsilon > 0$ for which $p(\cdot, x, y)$ is continuously differentiable on $[\theta_0-\epsilon, \theta_0 + \epsilon]$ for each $x,y\in S$. 
\end{itemize}
Set $\omega_\epsilon(x,y) = \sup_{|\theta-\theta_0|<\epsilon}|p'(\theta,x,y)|.$ Our next assumption involves a (uniform) minorization condition over the set $A$, which is standard in the theory of Harris recurrent Markov chains; see, for example, p.102 of \cite{meyn09}
\begin{itemize}
	\item[A5.] 
	There exists $\epsilon > 0$, a subset $A\subseteq S$, an integer $n\geq 1$, $\lambda>0$, and a probability $\varphi$ for which 
	\begin{equation*}
	P^n(\theta, x, dy) \geq \lambda \varphi(dy)
	\end{equation*}
	for $x\in A$, $y\in S$, and $|\theta-\theta_0| < \epsilon$.
\end{itemize}
For $a, b\in \R$, let $a\vee b \triangleq \max(a,b)$. We can now state our main theorem on differentiability of stationary expectations.
\begin{theorem}\label{thm:3}
	Assume that A4 and A5 hold. 
	Let $\kappa: \R_+ \to \R_+$ be a function for which $\kappa(x)\geq x$ and $\kappa(x)/x \to\infty$ as $x \to \infty$. 
	Suppose that there exist positive constants $\epsilon$, $c_0$, and $c_1$, and non-negative finite-valued functions $q$, $v_0$, and $v_1$ for which 
	\begin{align}
	(P(\theta)v_0)(x) &\leq v_0(x) - (\rvin{q(x)}\vee 1) + c_0 \I(x\in A), \label{eq:3.4}\\
	(P(\theta)v_1)(x) 
	&\leq 
	v_1(x) - 
	\kappa\left( \int_S ( 1 \vee \omega_\epsilon(x,y) ) (v_0(y) + 1) P(\theta, x, dy)\right)
	+ c_1 \I(x \in A),
	\label{eq:3.5}
	\end{align}
	for $x\in S$, $|\theta-\theta_0|<\epsilon$, and 
	\begin{equation}\label{bound:sup_v_0}
	\sup_{x\in A} v_0(x) < \infty
	\end{equation}
	Then:
	\begin{itemize}
		\item[(i)]
		There exists an open interval $\mathcal N$ containing $\theta_0$ for which $X$ is a positive recurrent Harris chain under $P(\theta)$ for each $\theta\in \mathcal N$;
		\item[(ii)]
		There exists a unique stationary distribution $\pi(\theta)$ satisfying $\pi(\theta) = \pi(\theta)P(\theta)$ for each $\theta\in \mathcal N$ and $\pi(\theta)\rvin{q} \leq c_0$ for $\theta\in \mathcal N$;
		\item[(iii)] 
		For each $f$ such that $|f(x)|\leq q(x) \vee 1$ for $x\in S$, there exists a solution $g$ (denoted $g=\Gamma f$) of Poisson's equation satisfying 
		\begin{equation*}
		((I-P(\theta_0))g)(x) = f(x) - \pi(\theta_0) f
		\end{equation*}
		for $x\in S$, and $|g(x)| = |(\Gamma f)(x)| \leq a(v_0(x) + 1)$ for $x \in S$, where $a$ is a finite constant;
		\item[(iv)]
		For each $f$ such that $|f(x)| \leq q(x) \rvin{\vee} 1$, $\alpha(\theta) = \pi(\theta)f$ is \what{continuously} differentiable at $\theta_0$, and
		\begin{equation}\label{C}
		\alpha'(\theta_0) = \int_S \pi(\theta_0, dx) \int_S p'(\theta_0, x,y) (\Gamma f) (y) P(\theta_0, x, dy).
		\end{equation}
	\end{itemize}
\end{theorem}
\begin{proof}{Proof.}
It is a standard fact that A5, (\ref{eq:3.4}), and (\ref{bound:sup_v_0}) imply that $X$ is a positive recurrent Harris chain under $P(\theta)$ for $\theta \in \mathcal N$ (where $\mathcal N$ is selected so that A5, (\ref{eq:3.4}) and (\ref{bound:sup_v_0}) are all in force); see, for example, \citet[p.313]{meyn09}. As a consequence, there exists a unique stationary distribution $\pi(\theta)$ for each $\theta \in \mathcal N$. Furthermore, (\ref{eq:3.4}) implies that the bound $\pi(\theta)\rvin{q}\leq c_0$ holds for $\theta\in \mathcal N$; see, for example, Corollary 4 of \cite{glynn08}. Because $X$ is Harris recurrent (and (\ref{eq:3.4}) holds), one can now invoke Theorem 2.3 of \cite{glynn96} to obtain (iii).

Turning to (iv), note that (\ref{eq:3.5}) guarantees that $\pi(\theta)\rvin{v_0} < \infty $ for $\theta\in \mathcal N$, so that $\pi(\theta)|\Gamma f| < \infty$. 
\cf{
Note that $\pi|f|<\infty$ is equivalent to $(\pi P)|f|<\infty$ and hence conditional Fubini applies: $\pi (P f) = (\pi P) f = \pi f$. Therefore,
$
\pi (Pf-f) = 0.
$
This validates (\ref{eq:3.3a}).}%
With the above conclusions having been verified, we can now appeal to (\ref{eq:3.3a}) to write
\begin{align}
\pi(\theta_0+h) f - \pi(\theta_0) f
&= \pi(\theta_0 + h) \big( P(\theta_0+h) - P(\theta_0)\big) \Gamma f\nonumber\\
&=\int_S \pi(\theta_0 + h, dx) \int_S \big(P(\theta_0+h, x, dy) - P(\theta_0, x, dy)\big)(\Gamma f)(y).\label{eq:3.6}
\end{align}
Set $s(x) = \int_S \omega_\epsilon(x,y) (v_0(y) + 1) P(\theta_0, x, dy)$ and put $\I_m(x) = \I(s(x) \geq m)$, $\I_m^c(x) = \I(s(x) < m)$. 
\cf{ Then $s(x)$ bound the finite difference 
	$$\frac{P_h-P}{h}\Gamma f \leq a s$$
	and from Glynn and Zeevi (2008) and (\ref{eq:3.5})
	\begin{align*}
	\pi(\theta) s 
	&\leq \int\pi(\theta,dx) \int(\omega_\epsilon(x,y)\vee 1)(v_0(y) + 1) P(\theta_0, x, dy)\\
	&\leq \int\pi(\theta,dx) \int\kappa\big((\omega_\epsilon(x,y)\vee 1)(v_0(y) + 1)\big) P(\theta_0, x, dy)\\
	&\leq c_1
	\end{align*}
}
Observe that since $|p(\theta_0+h, x, y) - p(\theta_0, x, y) |/h \leq \omega_\epsilon(x,y)$, and $\rvin{|}(\Gamma f)(y)\rvin{|} \leq a(v_0(y) +1 )$,
\begin{align}
&\int_S \pi(\theta_0 + h, dx) \I_m(x) \rvin{\bigg|}\left(\frac{P(\theta_0+h)-P(\theta_0)}{h} (\Gamma f)\right)(x)\rvin{\bigg|}\nonumber\\
&\leq \int_S \pi(\theta_0 + h, dx) \I_m(x) \int_S \omega_\epsilon(x,y) \,a(v_0(y)+1) P(\theta_0, x, dy)\nonumber\\
&\leq a\int_S \pi(\theta_0 + h, dx) \I_m(x) s(x)\nonumber\\
&\leq \frac{a}{\inf\{ \frac{\kappa(s(y))}{s(y)}: s(y)\geq m)\}} \int_S \pi(\theta_0+h, dx) \frac{\kappa(s(x))}{s(x)}s(x)\nonumber\\
&\leq \frac{a}{\inf\{ \frac{\kappa(s(y))}{s(y)}: s(y)\geq m)\}}\int_S \pi(\theta_0+h,dx) \kappa\left(\int_S (1\vee\omega_\epsilon(x,y))(v_0(y) + 1) P(\theta_0, x, dy)\right)\nonumber\\
&\leq \frac{a}{\inf\{ \frac{\kappa(s(y))}{s(y)}: s(y)\geq m)\}} c_1 \label{eq:3.7}
\end{align}
where the last inequality follows from (\ref{eq:3.5}) and Corollary 4 of \cite{glynn08}.

On the other hand, 
\begin{align*}
\int_S \pi(\theta_0+h,dx) \I_m^c (x) \frac{P(\theta_0+h)-P(\theta_0)}{h}(\Gamma f) (x)
\triangleq \int_S \pi(\theta_0+h, dx) s_h^m(x) = \pi(\theta_0+h)s_h^m,
\end{align*}
where 
\begin{align*}
|s_h^m(x) | \leq a\int_S \omega_\epsilon(x,y) (v_0(y)+1) P(\theta_0, x, dy)\I(s(x) < m)  \leq am,
\end{align*}
so $s_h^m$ is bounded. It follows that 
\begin{equation*}
\pi(\theta_0 + h) s_h^m - \pi(\theta_0)s_h^m = \pi(\theta_0 + h) (P(\theta_0+h)-P(\theta_0))(\Gamma s_h^m).
\end{equation*}
\cf{From (\ref{eq:3.4})
	$$\Gamma 1 \leq a(v_0+1)$$
	and hence
	$$\Gamma s_h^m \leq (am) (a(v_0+1))$$
}
Note that $\left|\frac{\Gamma s_h^m}{am}\right| \leq a(v_0+1)$ (because $\left| \frac{s_h^m}{am} \right|\leq q \vee 1$),
and hence
\begin{align}
|\pi(\theta_0 + h) s_h^m - \pi(\theta_0) s_h^m|
&\leq a^2 m|h| \int_S \pi(\theta_0+h,dx) \int_S \omega_\epsilon(x,y) P(\theta_0, x,dy) (v_0(y) + 1)\nonumber\\
&\leq a^2 m|h| \int_S \pi(\theta_0 + h, dx) s(x)\nonumber\\
&\leq a^2 m|h| c_1\to 0 \label{eq:3.9}
\end{align}
as $h\to 0$. Finally, 
\begin{equation*}
\int_S \pi(\theta_0, dx) s_h^m(x) = \int_S \pi(\theta_0, dx) \I_m^c(x)\int_S \frac{p(\theta_0+h, x, y)  - p(\theta_0, x, y)}{h} P(\theta_0, x, dy) (\Gamma f) (y)
\end{equation*}
and
\begin{equation*}
\frac{p(\theta_0+h, x, y) - p(\theta_0, x, y)}{h} \to p'(\theta_0, x, y)
\end{equation*}
as $h \searrow 0$. Furthermore, $|p(\theta_0+h, x, y) - p(\theta_0, x, y) |/h \leq \omega_\epsilon(x,y)$, $(\Gamma f)(y) \leq a(v_0(y) +1 )$, and 
\begin{align*}
\int_S \pi(\theta_0, dx) \int_S \omega_\epsilon(x,y) P(\theta_0, x, dy) (v_0(y) + 1) \leq \int_S \pi(\theta_0, dx) s(x) \leq c_1,
\end{align*}
so the Dominated Convergence Theorem implies that
\begin{equation}\label{eq:3.10}
\int_S \pi(\theta_0, dx) s_h^m(x) \to \int_S \pi(\theta_0, dx) \I_m^c(x) \cdot \int_S p'(\theta_0, x, y) P(\theta_0, x, dy) (\Gamma f)  (y)
\end{equation}
as $h\searrow 0$.

If we first let $h\to 0$ and then let $m \to \infty$, (\ref{eq:3.6}) through (\ref{eq:3.10}) imply part (iv) of our theorem.

\what{
Finally, turning to the continuity of the derivative, 
note that the exactly same argument as above gives
$
\alpha'(\theta_0+h) = \int_S \pi(\theta_0+h, dx) p'(\theta_0+h,x,y) \Gamma_{\theta_0+h} f(x) P(\theta_0,x,dy)
$ 
where $\Gamma_{\theta_0+h}f$ is the solution $g$ of the Poisson equation $g - P(\theta_0+h)g = f - \pi(\theta_0+h) f$.
Since
$$
\alpha'(\theta_0+h) - \alpha'(\theta_0)
=
\left(\alpha'(\theta_0+h) - \frac{\alpha((\theta_0+h)+(-h)) - \alpha(\theta_0+h)}{-h}\right) - \left(\alpha'(\theta_0) - \frac{\alpha(\theta_0+h)-\alpha(\theta_0)}{h}\right),
$$
and we have seen that the second term vanishes as $h\to 0$, we are done if we show that the first term also vanishes. 
Similarly as in (\ref{eq:3.3a}),
$
\alpha(\theta_0) - \alpha(\theta_0+h) = \pi(\theta_0) (P(\theta_0)-P(\theta_0+h)) \Gamma_{\theta_0+h}f
$. 
Therefore,
\begin{align}
&\alpha'(\theta_0+h) - \frac{\alpha((\theta_0+h)+(-h)) - \alpha(\theta_0+h)}{-h}
=
\alpha'(\theta_0+h) + \frac{\alpha(\theta_0) - \alpha(\theta_0+h)}{h}
\nonumber
\\
&
= 
\int_S \pi(\theta_0+h, dx) p'(\theta_0+h,x,y) \Gamma_{\theta_0+h} f(x) P(\theta_0,x,dy)
\nonumber
\\
&
\qquad
+
\int_S \pi(\theta_0, dx) \frac{p(\theta_0,x,y) - p(\theta_0+h,x,y)}{h} \Gamma_{\theta_0+h} f(x) P(\theta_0,x,dy)
\nonumber
\\
&
= 
\int_S \big(\pi(\theta_0+h, dx) - \pi(\theta_0, dx)\big) p'(\theta_0+h,x,y) \Gamma_{\theta_0+h} f(x) P(\theta_0,x,dy)
\label{eq:first_term_pf_t41}
\\
&
\qquad
+
\int_S \pi(\theta_0, dx) \left(p'(\theta_0+h,x,y)+\frac{p(\theta_0,x,y) - p(\theta_0+h,x,y)}{h}\right) \Gamma_{\theta_0+h} f(x) P(\theta_0,x,dy).
\label{eq:second_term_pf_t41}
\end{align}
Upon a perusal of the proof of Theorem 2.3 of \cite{glynn96}, one can see that the uniform majorization condition A5 and the uniform Lyapunov inequality (\ref{eq:3.4}) implies $|\Gamma_{\theta_0+h} f(x)| \leq a(v_0(x)+1)$ with the same constant $a$ as in (iii). One can prove that 
(\ref{eq:first_term_pf_t41}) vanishes as $h\to 0$ by the same argument as (\ref{eq:3.7}) and (\ref{eq:3.9}). 
On the other hand, 
(\ref{eq:second_term_pf_t41}) vanishes by the continuous differentiability condition A4 of $p$ and the dominated convergence along with (\ref{eq:3.5}). 
}
\end{proof}

As for Theorem~\ref{thm:1}, the proof also establishes a computable bound on $|\alpha'(\theta_0)|$, namely $|\alpha'(\theta_0)|\leq \rvin{a} c_1$ \rvin{where $a$ is the constant in (iii)}. 
\cf{maybe we want to explain how to compute $a$ or at least give a formula?} 
Also, as in Section~\ref{sec:random_horizon}, we can further simplify the condition when $X$ is the solution to the stochastic recursion (\ref{eqA}), in which the parameter $\theta$ affects only the distribution $Z_1$. When $p(\cdot,z)$ is continuously differentiable, (\ref{eq:3.5}) may be simplified as 
\begin{equation}
(P(\theta)v_1) (x) \leq v_1(x) - \kappa\left(\E ^{\theta_0}\Big(1\vee\sup_{|\theta-\theta_0|<\epsilon}|p'(\theta,Z_1) |\Big)\big(v_0(r(x,Z_1))+1\big)p(\theta,Z_1)\right) + c_1 \I(x\in A).
\end{equation}
With A5, (\ref{eq:3.4}), and (\ref{bound:sup_v_0}) also in force, this ensures the differentiability of $\alpha(\cdot)$ at $\theta_0$, with $\alpha'(\theta_0)$ given by 
\begin{equation}
\alpha'(\theta_0) = \int_S \pi(\theta_0, dx) \E^{\theta_0} (\Gamma f) (r(x,Z_1))p'(\theta_0,Z_1).
\end{equation}

A useful example on which to illustrate the above theory (and an important model in its own right) is that of the waiting time sequence $W=(W_n: n\geq 0)$ for the single-server G/G/1 queue, with first come first serve queue discipline. Let $V_n$ be the arrival time for the $\rvin{n}$\textsuperscript{th} customer, and let $\chi_{n+1}$ be the inter-arrival time that elapses between the arrival of the $\rvin{n}$\textsuperscript{th} and $(n+1)$\textsuperscript{st} customer. If $W_n$ is the waiting time (exclusive of service) for customer $n$, the $W_n$'s satisfy the stochastic recursion
\begin{equation}
W_{n+1} = [W_n + V_n - \chi_{n+1}]^+
\end{equation}
for $n\geq 0$, where $[x]^+ \triangleq \max(x,0).$ Assume that the $V_n$'s are iid, independent of the $\chi_n$'s (which are also assumed iid). Then, $W$ is a Markov chain taking values in $S = [0,\infty)$. It is well known that $W$ is a positive recurrent Harris chain if $\E V_0 < \E \chi_1$, and that $\E V_0^{p+1}\rvin{<\infty}$ is then a necessary and sufficient condition for guaranteeing the finiteness of $\pi f_p$, where $f_p(x) = x^p$ (with $p>0$); 
see, for example, \cite{kiefer1956}. 
This suggests that it then typically will be the case that the $p$\textsuperscript{th} moment should be differentiable when $\E V_0 ^{p+1}< \infty$. 

We consider this problem in the special case in which the service times are finite mean Pareto random variables (rv's), and $\theta$ influences the scale parameter of the Pareto distribution. In other words, we consider the setting in which 
\begin{equation*}
P^\theta(V_0 > v) = (1+\theta v)^{-\alpha}
\end{equation*} 
for $\alpha > 1$. 
\cf{$\E^\theta v = \frac{1}{\theta}\frac{1}{\alpha-1}$}
In this case, the density of $V_0$ under $P^\theta$ is given by $\theta h_V(\theta v),$ where $h_V(v) = \alpha(1+v)^{-\alpha -1}$, so that 
\begin{equation*}
p(\theta, v) = \left(\frac{\theta}{\theta_0}\right) \left(\frac{1+\theta v}{1+\theta_0 v}\right)^{-\alpha-1}
\end{equation*}
and
\begin{equation*}
p'(\theta, v) = p(\theta, v) \left(\frac{1}{\theta}  - (\alpha + 1)\frac{v}{(1 + \theta v)}\right).
\end{equation*}
Note that both the density $p$ and its derivative (with respect to $\theta$) are bounded functions. Furthermore, the rv $p'(\theta_0, V_i)$ has mean zero under $P^{\theta_0}$. For any $c>0$, the set $A=[0,c]$ is easily seen to satisfy condition A5, and A4 is trivially verified (with $\omega_\epsilon(\cdot)$ bounded). Then, if $v_0(x) = a_1 x^{p+1}$, $v_1(x) = a_2 x^{r+2}$, and $\kappa(x)=x^{\frac{1+r}{1+p}}$ (with $r>p$ and $a_1$, $a_2$ chosen suitably), we see that (\ref{eq:3.4}), (\ref{eq:3.5}), and (\ref{bound:sup_v_0}) all hold, guaranteeing the differentiability of $\pi(\theta)f_p$ (according to Theorem~\ref{thm:3}). 

For example, to verify (\ref{eq:3.4}), we note that
\begin{equation*}
x^{-p}\big[(P(\theta)v_0)(x) - v_0(x)\big] = a_1 x \E ^{\theta_0} \left(\left[1+\frac{1}{x} \left(\frac{\theta_0}{\theta} V_0 - \chi_1\right)\right]^+\right)^{p+1} - a_1 x.
\end{equation*}
Observe that as $x \to \infty$,
\begin{align*}
&x f_{p+1}\left(\left[1+\frac 1 x \left(\frac{\theta_0}{\theta} V_0 - \chi_1\right)\right]^+\right) - x\\
&= x \left(f_{p+1}(1) + f_{p+1}'(1)\left(\frac{1}{x}\right)\left(\frac{\theta_0}{\theta} V_0 - \chi_1\right)\right) - x + o(1) \qquad a.s.\\
&= \rvin{(}p\rvin{+1)}\left(\frac{\theta_0}{\theta} V_0 - \chi_1\right) + o(1) \qquad a.s.
\end{align*}
\rvout{as $x \to \infty$,} where $o(1)$ represents a function $k(x)$ such that $k(x) \to 0$ as $x \to \infty$ \rvin{uniformly in a neighborhood of $\theta_0$}. 
In addition, note that for $p>0$ and $x>0$, the mean value theorem implies that $f_{p+1}(1+x) = f_{p+1}(1) + f'_{p+1}(1+\xi)x$ for some $\xi \in [0,x]$, so that $f_{p+1}(1+x) = f_{p+1}(1) + (p+1) (1+\xi)^{p} x \leq 1 + (p+1)(1+x)^px$. Consequently, 
\begin{align*} 
\hcancel{
&\left|\ 
	x\left(
		\left[
			1+\frac{1}{x}
			\left(
				\frac{\theta_0}{\theta} V_0 - \chi_1
			\right)
		\right]^+
	\right)^{p+1} - x\ 
\right|
\ 
\\ 
}
&
	x\left(
		\left[
			1+\frac{1}{x}
			\left(
				\frac{\theta_0}{\theta} V_0 - \chi_1
			\right)
		\right]^+
	\right)^{p+1} - x\ 
\ 
\\ 
&\leq x 
\left(
	1+\frac 1 x\frac{\theta_0}{\theta} V_0
\right)^{p+1} - x\\
&\leq x\left(1 + (p+1) \left(1 + \frac{1}{x}\frac{\theta}{\theta_0}V_0\right)^p \frac{1}{x}\frac{\theta}{\theta_0} V_0\right) - x\\
&\hcancel{
	\leq (p+1) \left(1 + \frac{1}{x}\frac{\theta}{\theta_0}V_0\right)^{p+1}
}\\
&\leq (p+1) \left(1 + \frac{1}{x}\frac{\theta}{\theta_0}V_0\right)^{p}\frac{\theta}{\theta_0}V_0
\end{align*}
\cf{
	\rvin{
	and
	\begin{align*} 
	&
	x-
		x\left(
			\left[
				1+\frac{1}{x}
				\left(
					\frac{\theta_0}{\theta} V_0 - \chi_1
				\right)
			\right]^+
		\right)^{p+1} 
	\\ 
	&\leq x - x\left(
			\left[
				1-\frac{\chi_1}{x}
			\right]^+
		\right)^{p+1} \\
	&\leq x - x\left(
				1-\frac{\chi_1}{x}
		\right)^{p+1} 
	\end{align*}
	}
}
\cf{For $p \in (-1,0)$, mean value theorem implies
	$f_{p+1}(1+x) = f_{p+1}(1) + f'_{p+1}(1+\xi)x$ for some $\xi \in [0,x]$, so that $f_{p+1}(1+x) = f_{p+1}(1) + (p+1) (1+\xi)^{p} x \leq 1 + (p+1)x$. Consequently, 
\begin{align*} 
&
x\left(
	\left[
		1+\frac{1}{x}
		\left(
			\frac{\theta_0}{\theta} V_0 - \chi_1
		\right)
	\right]^+
\right)^{p+1} - x\  
\\ 
&\leq x 
\left(
	1+\frac 1 x\frac{\theta_0}{\theta} V_0
\right)^{p+1} - x\\
&\leq x\left(1 + (p+1)  \frac{1}{x}\frac{\theta}{\theta_0} V_0\right) - x\\
&\leq (p+1)\frac{\theta}{\theta_0}V_0
\end{align*}

and
\begin{align*} 
&
x-
x\left(
	\left[
		1+\frac{1}{x}
		\left(
			\frac{\theta_0}{\theta} V_0 - \chi_1
		\right)
	\right]^+
\right)^{p+1} 
\\ 
&\leq x - x\left(
		\left[
			1-\frac{\chi_1}{x}
		\right]^+
	\right)^{p+1} 
\leq x - x\left(
		\left[
			1-\frac{\chi_1}{x}
		\right]^+
	\right) 
\leq x - x\left(
			1-\frac{\chi_1}{x}
	\right)
\leq \chi_1
\end{align*}
}
Since $\E V_0^{p+1}< \infty$, \rvout{the Dominated Convergence theorem ensures} \rvin{Fatou's lemma applies to ensure} that
\hcancel{
\begin{equation*}
\E^{\theta_0} \left(x f_{p+1} \left(\left[1+\frac{1}{x}\left(\frac{\theta_0}{\theta}V_0 - \chi_1\right)\right]^+\right)-x\right) \to (p+1)\E ^{\theta_0} \left(\frac{\theta_0}{\theta}V_0 - \chi_1\right) = (p+1)\left(\frac1 {\theta\rvin{(\alpha - 1)}} - \E \chi_1\right)
\end{equation*}
}
\rvin{
	\begin{align*}
	\limsup_{x\to\infty}\,\sup_{\theta} \E^{\theta_0}  \left(x f_{p+1} \left(\left[1+\frac{1}{x}\left(\frac{\theta_0}{\theta}V_0 - \chi_1\right)\right]^+\right)-x\right) 
	&\leq (p+1)\E ^{\theta_0}\sup_{\theta} \left(\frac{\theta_0}{\theta}V_0 - \chi_1\right) \\
	&= (p+1)\sup_{\theta}\left(\frac1 {\theta(\alpha - 1)} - \E \chi_1\right)
	\end{align*}
}
as $x\to\infty$ (with convergence that is uniform in a neighborhood of $\theta_0$). If we choose $a_1$ so that $a_1 \rvin{(}p\rvin{+1)\sup_{\theta}}(\frac{1}{\theta\rvin{(\alpha - 1)}} - \E \chi_1) \leq -2$ \rvout{(uniformly in $\theta$)} and $c$ so that
\begin{equation*}
a_1\rvin{\sup_{\theta}}\E^{\theta_0} \left(xf_{p+1}\left(\left[1+\frac{1}{x}\left(\frac{\theta_0}{\theta}V_0 - \chi_1\right)\right]^+\right)-x\right) \leq -1
\end{equation*}
for $x\geq c$, then (\ref{eq:3.4}) is validated. A similar argument applies to (\ref{eq:3.5}), in view of the boundedness of $\omega_\epsilon(\cdot)$. Our argument therefore establishes that $\pi f_p$ is differentiable if $\E V_0 ^q < \infty$ for some $q>p+2$. This is not quite the ``correct'' result (in that we previously argued that $\E V_0^{p+1}<\infty$ should be sufficient.)

The reason that our argument fails to provide optimal condition here has to do with special random walk structure that is present in the process $W$ that is difficult for general machinery to exploit. The challenge arises at (\ref{eq:3.3a}) above. Note that the argument just provided for $W$ involves using $v_0=a_1 f_{p+1}$ as a bound on the solution $g$ to Poisson's equation for $f_p$. (As we shall see in a moment, $g$ is indeed exactly of order $x^{p+1}$). The problem is that neither $P(\theta_0+h) f_{p+1}$ nor $P(\theta_0)f_{p+1}$ in (\ref{eq:3.3a}) are integrable with respect to $\pi(\theta_0+h)$ unless $\E V_0 ^{p+2}<\infty$. This is what leads to the extra moment appearing in our argument for $W$ above. Thus, any argument that yields differentiability under the hypothesis $\E V_0^{p+1}<\infty$ must take advantage of the fact that the random walk structure of $W$ yields the integrability of $(P(\theta_0 + h) - P(\theta_0))g$ under $\E V_0^{p+1}<\infty$ without demanding the integrability of $P(\theta_0)g$ and $P(\theta_0+h)g$ separately. 

It is shown in \cite{glynn96} that, in view of the fact that $W$ regenerates at hitting times of $\{0\}$, the solution $g$ to Poisson's equation for $f_p$ can be expressed as 
\begin{equation}\label{A}
g(x) = \E_x ^{\theta_0} \sum_{j=0}^{\tau(0)-1} (f_p(W_j)-\pi(\theta_0)f_p),
\end{equation}
where $\tau(0) = \inf\{n\geq 1: W_n = 0\}$ is the hitting time of $\{0\}$. Let $Z_j = V_{j-1} - \chi_j$, $S_j = Z_1 + \cdots + Z_j$, (for $j\geq 1$), $\tau_x(0) = \inf\{j\geq 1: x + S_j \leq 0\}$, $\mu = \E Z_1$, and note that (\ref{A}) implies that
\begin{align}
(P(\theta_0 + h)g)(x) - (P(\theta_0)g)(x)
&= \E^{\theta_0} g(W_1) [p(\theta_0 + h, V_0)-1]\label{B}\\
&= \E^{\theta_0} \sum_{j=1}^{\tau_x(0)\rvin{-1}} [(x+S_j)^p - \pi(\theta_0) f_p] (p(\theta_0 + h, V_0) -1 ) \I(x + Z_1 > 0).\nonumber
\end{align}
But
\begin{align*}
\sum_{j=1}^{\tau_x(0)\rvin{-1}} (x+S_j)^p (p(\theta_0+h,V_0)-1)
&= x^p\sum_{j=1}^{\tau_x(0)\rvin{-1}} [(1+\frac{S_j-V_0}{x})^p+p\xi_j(x)^{p-1}\frac{V_0}{x}](p(\theta_0 + h, V_0)-1)
\end{align*}
where $\xi_j(x)$ lies between $1+S_j/x-V_0/x$ and $1+S_j/x$. It is easily argued, based on Riemann sum approximations, that
\begin{align*}
\sum_{j=1}^{\tau_x(0)\rvin{-1}} \xi_j(x)^{p-1} \frac{1}{x} 
&\to \int_0^{\rvin{1/}|\mu|} (1+\mu s)^{p-1}ds \qquad a.s.\\
&=\frac{1}{|\mu|}\cdot \frac{1}{p}
\end{align*}
as $x \to \infty$. Furthermore, $p(\theta_0+h,V_0)-1$ is a mean zero rv that is independent of $(1+(S_j-V_0)/x)^p$ for $j\geq 1$ and $\E \tau_x(0) \sim x / |\mu|$ as $x\to \infty$ (where $a_1(x) \sim a_2(x)$ as $x\to \infty$ means that $a_1(x)/a_2(x) \to 1$ as $x \to \infty$). In view of (\ref{B}), this suggests that 
\begin{equation*}
(P(\theta_0+h)g)(x) - (P(\theta_0)g)(x) \sim 
\frac{x^p}{|\mu|} \E V_0 (p(\theta_0+h,V_0)-1)
\end{equation*}
as $x\to \infty$ (i.e., one power lower than the growth of $g$ itself). Thus, this style of argument can successfully deal with the integrability issue discussed earlier, and leads to a validation of the derivative formula (\ref{C}) for $W$ under the assumption $\E V_0^{p+1}<\infty$. 
A rigorous statement and the remaining details of the proof can be found in the Appendix. 
This differentiation result for $W$ can also be found in \cite{heidergott09}, with a different (and longer) proof, and with some steps that appear to be incomplete. (In particular, the paper asserts that $\rvin{\E_x ^{\theta_0} \sum_{j=0}^{\tau(0)-1} f_p(W_j)}$ is bounded \rvin{for any fixed $\theta$ and $p$, which implies that our function $g$ grows at most linearly regardless of $p$}; see p.248).

\appendix

\section{The G/G/1 Queue Example.}\label{appendix:A}
In this section, we prove the following statement: if $\P^{\theta}(V_0 > v) = (1+\theta v)^{-r-2}$, 
\cf{this means that $V_0$ has almost $(r+1)$-st moment (for all $\theta$) and $\pi(\theta)$ has $r$-th moment.}%
then $\alpha(\theta) = \pi(\theta)f_p$ is differentiable for $1 \leq p < r$, and the derivative is 
\begin{equation}\label{eq:a.0}
\alpha'(\theta_0) 
= \E_{\pi(\theta_0)}^{\theta_0} p'(\theta_0,V_0)\Gamma f_p(W_1). 
\end{equation}
It turns out to be handy to have the following bound. The proof of the claim will be provided at the end of this section. 
\begin{claim} Let $f_{p;m} \triangleq m\wedge f_p$.There is a constant $d>0$ and $h_0>0$ such that
\begin{equation}\label{eq:a.1}
|(P(\theta_0+h)\Gamma f_{p;m}-P(\theta_0)\Gamma f_{p;m})(x)| \leq hd(x^p+1)
\end{equation}
for $h<h_0$ and $m \in [0,\infty]$.
\end{claim}
First note that (\ref{eq:3.4}) can be established as in Section~\ref{sec:equilibrium} with $v_0(x) = x^{p+1}$ and $q(x)= x^p$ for any $p < r$, and hence, $|\Gamma f_p(x)| \leq c(x^{p+1}+1)$  for $p<r$ by \cite{glynn96}, and $f_p$ is $\pi(\theta)$-integrable for $p<r$ by \cite{glynn08}. 
Since $\Gamma f_{p;m}$ is $\pi(\theta_0+h)$-integrable (since it is bounded by an affine function), \cf{in fact, this is not the case for $p<1$} if we let $\alpha_m(\theta) \triangleq \pi(\theta)f_{p;m}$, then $\alpha_m(\theta_0+h) - \alpha_m(\theta_0) = \pi(\theta_0+h)(P(\theta_0+h)-P(\theta_0))\Gamma f_{p;m}$. 
Monotone convergence theorem guarantees that $\alpha_m(\theta_0+h) - \alpha_m(\theta_0)$ converges to $\alpha(\theta_0+h) - \alpha(\theta_0)$ as $m\to\infty$; on the other hand, applying bounded convergence twice along with \cite{glynn08} and (\ref{eq:a.1}), one can conclude that $\pi(\theta_0+h)(P(\theta_0+h)-P(\theta_0))\Gamma f_{p;m}\to \pi(\theta_0+h)(P(\theta_0+h)-P(\theta_0))\Gamma f_{p}$. Therefore, (\ref{eq:3.3a}) is valid.
Now, set $s(x) = x^p+1$ and put $\I_m(x) = \I(x\geq m)$ and $\I_m^c = \I(x <m)$. 
Then,
\begin{align}
\int_{\R_+}\pi(\theta_0+h,dx) \I_m(x) \frac{P(\theta_0+h) - P(\theta_0)}{h}\Gamma f_p(x)
&\leq d\int\pi(\theta_0+h,dx)\I_m(x)s(x)\nonumber\\
&\leq d\int\pi(\theta_0+h,dx)\frac{x^{p+\epsilon} + 1}{s(x)} s(x) \frac{m^p+1}{m^{p+\epsilon}+1}\nonumber\\
&= cd\frac{m^p+1}{m^{p+\epsilon}+1}\label{eq:a.2}
\end{align}
for $0<\epsilon <r-p$ and some constant $c>0.$ On the other hand, let
\begin{equation*}
\int_{\R_+} \pi(\theta_0+h,dx)\I_m^c(x) \frac{P(\theta_0+h)-P(\theta_0)}{h} \Gamma f_p(x) \triangleq \pi(\theta_0+h)s_h^m,
\end{equation*}
then, $s_h^m$ is bounded by $c(m^{p}+1)$ for some $c$, and hence $\Gamma s_h^m \leq a(m^{p}+1)(x+1)$ for some $a>0$. 
	\cf{
	\begin{align*}
	s_h^m(x) 
	&= \I_m^c(x) \frac{P(\theta_0+h) - P(\theta_0)}{h}\Gamma f_p(x)\\
	&= \I_m^c(x) \E_x^{\theta_0}\frac{p(\theta_0+h,V_0) - 1}{h}\Gamma f_p(W_1)	
	\end{align*}
	}
Therefore, by the same argument as in (\ref{eq:3.3a}),
\begin{align}
|\pi(\theta_0+h)s_h^m - \pi(\theta_0)s_h^m| 
&=|\pi(\theta_0+h)(P(\theta_0+h) - P(\theta_0))(\Gamma s_h^m)|\nonumber\\
&= \int \pi(\theta_0+h,dx) \E_x \Gamma s_h^m([x+V_0-\chi_1]^+)(p(\theta_0+h,V_0)-1)\nonumber\\
&
\leq hd'(m^p+1).\label{eq:a.3}
\end{align}
for some $d'>0$.
Finally,
\begin{equation*}
\pi(\theta_0)s_h^m = \int_{\R_+} \pi(\theta_0,dx)\I_m^c(x) \E_x^{\theta_0} \frac{p(\theta_0+h,V_0)-1}{h}\Gamma f_p(W_1).
\end{equation*}
For each $x$, $s_h^m(x) \to \I_m^c(x)\E_x^{\theta_0} p'(\theta_0+h,V_0)\Gamma f_p(W_1)$ as $h\to 0$ by bounded convergence.
Also, due to the definition of $\I_m^c(x)$ and boundedness of $p'$, $s_h^m$ itself is bounded w.r.t.\ $h$ and $x$.
Therefore, 
applying bounded convergence theorem we conclude that 
\begin{equation}\label{eq:a.4}
\pi(\theta_0)s_h^m \to \int_{\R_+}\pi(\theta_0,dx)\I_m^c(x)\E_x^{\theta_0} p'(\theta_0,V_0)\Gamma f_p(W_1)
\end{equation}
as $h\to 0$.
Therefore, if we let $h\to0$ and then let $m\to \infty$, (\ref{eq:a.2}), (\ref{eq:a.3}), and (\ref{eq:a.4}) imply (\ref{eq:a.0}). 

\begin{proof}{Proof of the claim.}
First note that 
\begin{align*}
P(\theta_0 + h) \Gamma f_{p;m}(x) - P(\theta_0) \Gamma f_{p;m}(x) 
&= \E^{\theta_0} \sum_{j=1}^{\tau_x(0)-1} [m\wedge(x+S_j)^p - \pi(\theta_0)f_{p;m}] (p(\theta_0+h,V_0)-1)
\end{align*}
Set $\sigma_x(0) = \inf \{n\geq 1: x + S_n - V_0 \leq 0\}$, then obviously $\sigma_x(0) \leq \tau_x(0)$.
Considering the Taylor expansion of $f_p$ up to $\lfloor p\rfloor $\textsuperscript{th} term, 
\begin{align*}
&
\sum_{j=1}^{\sigma_x(0)-1}m\wedge(x+S_j)^p
\leq \sum_{j=1}^{\sigma_x(0)-1}(x+S_j)^p
\\
&= 
\sum_{j=1}^{\sigma_x(0)-1} \left(\sum_{n=0}^{\lfloor p\rfloor-1}c_n(x+S_j-V_0)^{p-n}V_0^n
+ c_{\lfloor p\rfloor} (x+S_j - V_0)^{p-\lfloor p \rfloor}V_0^{\lfloor p \rfloor} + c_{\lfloor p \rfloor}(V_{0,j}^*)^{p-\lfloor p \rfloor}V_0^{\lfloor p \rfloor}\right)
\\
&\leq 
\sum_{j=1}^{\sigma_x(0)-1}(x+S_j-V_0)^p +
\sum_{n=1}^{\lfloor p\rfloor}\sum_{j=1}^{\sigma_x(0)-1}c_n(x+S_j-V_0)^{p-n}V_0^n + c_{\lfloor p \rfloor} V_0^p(\sigma_x(0)-1)
\end{align*}
\cf{
\begin{align*}
&
\sum_{j=1}^{\sigma_x(0)-1}m\wedge(x+S_j)^p
\leq \sum_{j=1}^{\sigma_x(0)-1}(x+S_j)^p
\\
&= 
\sum_{j=1}^{\sigma_x(0)-1} \left(\sum_{n=0}^{\lfloor p\rfloor-1}c_n(x+S_j-V_0)^{p-n}V_0^n
+ c_{\lfloor p\rfloor} (x+S_j - V_0 + V_{0,j}^*)^{p-\lfloor p \rfloor}V_0^{\lfloor p \rfloor}\right)
\\
&= 
\sum_{j=1}^{\sigma_x(0)-1} \left(\sum_{n=0}^{\lfloor p\rfloor-1}c_n(x+S_j-V_0)^{p-n}V_0^n
+ c_{\lfloor p\rfloor} (x+S_j - V_0)^{p-\lfloor p \rfloor}V_0^{\lfloor p \rfloor} + c_{\lfloor p \rfloor}(V_{0,j}^*)^{p-\lfloor p \rfloor}V_0^{\lfloor p \rfloor}\right)
\\
&\leq 
\sum_{n=0}^{\lfloor p\rfloor}\sum_{j=1}^{\sigma_x(0)-1} c_n(x+S_j-V_0)^{p-n}V_0^n + c_{\lfloor p \rfloor} V_0^p(\sigma_x(0)-1)
\\
&\leq 
\sum_{j=1}^{\sigma_x(0)-1}(x+S_j-V_0)^p +
\sum_{n=1}^{\lfloor p\rfloor}\sum_{j=1}^{\sigma_x(0)-1}c_n(x+S_j-V_0)^{p-n}V_0^n + c_{\lfloor p \rfloor} V_0^p(\sigma_x(0)-1)
\end{align*}
}%
where $c_n = \frac{p(p-1)\cdots(p-n)}{n!}$ and $0\leq V_{0,j}^* \leq V_0.$ 
Since $\sigma_x(0)$ and $(x+S_j-V_0)$ are independent of $V_0$, and $p(\theta_0+h,V_0) - 1$ is a mean zero rv, 
$\E^{\theta_0}\sum_{j=1}^{\sigma_x(0)-1}(x+S_j-V_0)^p (p(\theta_0+h),V_0) - 1) = 0$ and $\E^{\theta_0}\sum_{j=1}^{\sigma_x(0)-1}\pi(\theta_0)f_{p,m}(p(\theta_0+h),V_0) - 1) = 0$, 
and hence
\begin{align*}
&\E^{\theta_0}\sum_{j=1}^{\sigma_x(0)-1}[m\wedge (x+S_j)^p - \pi(\theta_0)f_{p;m}](p(\theta_0+h,V_0) - 1) \\
&\leq \sum_{n=1}^{\lfloor p\rfloor}c_n\E V_0^n(p(\theta_0+h,V_0)-1)\E\sum_{j=1}^{\sigma_x(0)-1}(x+S_j-V_0)^{p-n} + c_{\lfloor p \rfloor} \E (\sigma_x(0)-1)\E V_0^p(p(\theta_0+h,V_0) - 1).
\end{align*}
\cf{No, this is not right. The sign of $(p(\theta_0+h,V_0)-1)$ is not always positive.}
Note that for $s\leq p$
\begin{equation}\label{eq:a.6}
0\leq \E^\theta_0 \sum_{j=1}^{\sigma_x(0)-1}(x+S_j-V_0)^{s} 
\leq \E^\theta_0 \sum_{j=0}^{\tau_x(0)-1}(x+S_j)^{s} 
= \Gamma f_s + \pi(\theta_0)f_s\E \tau_x(0) 
\leq c_{s+1}(x^{s+1}+1).
\end{equation}
\cf{
\begin{equation}
0\leq \E^\theta_0 \sum_{j=1}^{\sigma_x(0)-1}(x+S_j-V_0)^{s} 
\leq \E^\theta_0 \sum_{j=1}^{\sigma_x(0)-1}(x+S_j)^{s} 
\leq \E^\theta_0 \sum_{j=0}^{\tau_x(0)-1}(x+S_j)^{s} 
= \Gamma f_s + \pi(\theta_0)f_s\E \tau_x(0) 
\leq c_{s+1}(x^{s+1}+1).
\end{equation}
}%
Therefore, $\E^{\theta_0}\sum_{j=1}^{\sigma_x(0)-1}[(x+S_j)^p - \pi(\theta_0)f_p](p(\theta_0+h,V_0) - 1) = O(hx^p)$. On the other hand, 
\begin{align*}
\left|\E \sum_{j=\sigma_x(0)}^{\tau_x(0)-1} (x+S_j)^p\right| 
&\leq \left|\E \sum_{j=0}^{\tau_V(0)-1} (V+S_j)^p\right|
\leq \left|\E \sum_{j=0}^{\tau_V(0)-1} (V+S_j)^p-\pi(\theta_0)f_{p-1}\right| 
+ \left|\E \sum_{j=0}^{\tau_V(0)-1}\pi(\theta_0)f_{p-1}\right| \\
&\leq |\E \Gamma f_{p-1}(V)| + |\pi(\theta_0)f_{p-1}\E \gamma(V)| < \infty
\end{align*}
where $\gamma(x) \triangleq \E \tau_x(0)$. Likewise, $\E \sum_{j=\sigma_x(0)}^{\tau_x(0)-1}\pi(\theta_0)f_p$ can be bounded by constant, and the conclusion of the claim follows. 
\end{proof}

 \section{Completeness of weighted normed spaces.}\label{appendix:B}
$L_w$, $\mathcal L_w$, and $\mathcal M_w$ are obviously linear spaces, and it is easy to see that the associated norms are legitimate norms. 
The completeness of $L_w$ is an immediate consequence of the completeness of $L_e$ (where $e(x) \equiv 1$) and that if $\{h_n\}_{n=1,\ldots}$ is a Cauchy sequence in $L_w$, then $\{h_n/w\}_{n=1,\ldots}$ is a Cauchy sequence in $L_e$, along with the fact that the point-wise limit of a measurable function is measurable. 
To see that $\mathcal L_w$ is also complete, suppose that $\{Q_n\}_{n=1,\ldots}$ is a Cauchy sequence in $\mathcal L_w$.
Then, $\{Q_n f\}_{n=1,\ldots}$ is Cauchy in $L_w$ for any fixed $f$. 
The completeness of $L_w$ guarantees that there exists $\phi_f\in L_w$ such that $\|Q_nf-\phi_f\|_w \to 0$. Define $Q$ so that $Qf \triangleq \phi_f$. 
From this construction, $Q$ is obviously a linear operator.
Now, to show that $Q$ is the limit of $Q_n$ w.r.t. $\vertiii{\cdot}_w$, note that for any given $\epsilon>0$, one can choose $N$ such that $n,m\geq N$ implies $\|Q_nf - Q_mf\|_w \leq \epsilon$ for all $f$ such that $\|f\|_w=1$. 
Noting that
$$
\|Q_n f - Qf\|_w \leq \|Q_nf - Q_mf\|_w + \|Q_mf - Qf\|_w \leq \epsilon + \|Q_mf - Qf\|_w,
$$
and taking $m\to\infty$, one concludes that $n\geq N$ implies 
$$
\|Q_n f - Qf\|_w\leq \epsilon \text{ for all } f \text{ such that } \|f\|_w = 1.
$$
That is, $Q_n \to Q$ in $\vertiii{\cdot}_w$.
An immediate consequence is that $Q$ is bounded, and hence, we are left with showing that $Q$ is a genuine kernel. The measurability of $Q(\cdot,A)$ (for each $A\subseteq B$) is obvious since $Q(\cdot,A)$ is a pointwise limit of $Q_n(\cdot,A)$.
To show that $Q(x,\cdot)$ is sigma additive (and hence, it is indeed a measure) for each fixed $x$, suppose that $\{E_i\}_{i=1,\ldots}$ is a countable collection of disjoint measurable sets.
\begin{align*}
Q(x,\cup_i E_i) 
&
= \lim_{n\to\infty} Q_n(x,\cup_i E_i)
= \lim_{n\to\infty} \sum_i Q_n(x,E_i)
= \sum_i\lim_{n\to\infty} Q_n(x,E_i)
= \sum_i Q(x,E_i),
\end{align*}
\sloppy where the third equality is from bounded convergence along with the fact that 
$\sum_{i}|Q_n(x,E_i)| = \int \sum_i \text{sgn} (Q_n(x,E_i)) \I_{E_i}(y)Q_n(x,dy)$ is bounded (since $\|\sum_i\text{sgn}(Q_n(x,E_i))\I_{E_i}(\cdot)\|_w \leq 1$ and $Q_n$ is convergent in $\vertiii{\cdot}_w$) for sufficiently large $n$'s.
Since $0 = Q_n(x,\emptyset) \to Q(x,\emptyset)$, $Q$ is indeed a kernel.
The completeness of $\mathcal M_w$ follows from a similar (but easier) argument.

\bibliographystyle{apalike}
\bibliography{LyapunovConditions_ref}

\end{document}